%% file: AdvectionReaction_v3.tex
\documentclass[onefignum,onetabnum]{siamart190516}
\input{ex_shared}

\ifpdf
\hypersetup{
  pdftitle={Adaptive Dnetwork},
  pdfauthor={}
}
\fi

\DeclareUnicodeCharacter{2212}{-}
\usepackage[utf8]{inputenc}
\usepackage{bm}
\usepackage{mathtools}
\usepackage{indentfirst}
\usepackage{wrapfig}
\usepackage{tcolorbox}
\usepackage{color}
\usepackage[export]{adjustbox}
\usepackage{makecell}
\usepackage{textcomp}
\usepackage{algorithm }
\usepackage{algorithmic}
\usepackage{booktabs}
\usepackage{subcaption} 

\newcommand{\R}{\mathbb{R}}

\usepackage{graphicx}
\usepackage{diagbox}
\usepackage{pifont}
\newcommand{\vertiii}[1]{{\left\vert\kern-0.25ex\left\vert\kern-0.25ex\left\vert #1 
    \right\vert\kern-0.25ex\right\vert\kern-0.25ex\right\vert}}

\newcommand{\bSigma}{\mbox{\boldmath${\Sigma}$}}

\newcommand{\bomega}{\mbox{\boldmath${\omega}$}}
\newcommand{\bxi}{\mbox{\boldmath$\xi$}}

\newcommand{\bff}{{\bf f}}

\newcommand{\bH}{{\bf H}}

\newcommand{\bb}{{\bf b}}

\newcommand{\bx}{{\bf x}}

\newcommand{\cM}{{\cal M}}

\newcommand{\cT}{{\cal T}}

\def\ba{{\bf a}}
\def\bb{{\bf b}}
\def\bc{{\bf c}}
\def\bl{{\bf l}}

\def\br{{\bf r}}
\def\bx{{\bf x}}
\def\by{{\bf y}}

\def\cM{{\cal M}}

\def\cT{{\cal T}}

\begin{document}

\maketitle
\begin{abstract}
The least-squares neural network (LSNN) method introduced in \cite{Cai2021linear} for linear advection-reaction equations is capable of accurately approximating discontinuous solutions without {\it a priori} knowledge of the interface location. However, the resulting discretization is a non-convex optimization problem that is computationally intensive and complex. In this paper, we propose a structure-guided Gauss-Newton (SgGN) method that alternates between the linear (output) and the nonlinear (hidden layer) parameters. At each outer iteration, the linear parameters are computed by a linear solver, and the nonlinear parameters are updated by a modified Gauss-Newton (GN) method that explicitly removes the singularities of the GN matrix.
Numerical experiments for all test problems presented in \cite{Cai2021linear} show that the SgGN method is superior to the Adam optimizer \cite{kingma2015}, the commonly used first-order optimization algorithm, not only in computational cost but, more importantly, in accuracy. 

\end{abstract}

\begin{keywords} Advection-reaction equation, Least-squares method, ReLU neural network, Gauss-Newton method\end{keywords}

\section{Introduction}\label{s:introduction}

Let $\Omega$ be a bounded open domain in ${\R}^2$, and let $\bm{\beta}(\bx) = (\beta_1, \beta_2)^T\in C^1(\bar{\Omega})^2$ and $\gamma \in C(\bar{\Omega})$ be the given advective velocity field and reaction coefficient, respectively. Denote by
\[
\Gamma_- = \{\bx\in\Gamma :\, \bm{\beta}(\bx)\! \cdot\! \bm{n}(\bx) <0\}
\]
the inflow part of the boundary $\Gamma=\partial \Omega$, where $\bm{n}(\bx)$ is the unit outward vector normal to $\Gamma$ at $\bx\in \Gamma$. Without loss of generality, assume that the magnitude of ${\bm \beta}(\bx)$ is one in $\Omega$, i.e., $|{\bm \beta}(\bx)| = 1$ for $\bx \in \Omega$. Consider the following linear advection-reaction equation 
\begin{equation}\label{pde1}
    \left\{\begin{array}{rccl}
    u_{\bm\beta} + \gamma u &= & f &\text{ in }\,\, \Omega, \\
    u&=&g &\text{ on }\,\, \Gamma_{-},
    \end{array}\right.
\end{equation}
where $f \in L^2(\Omega)$ and $g \in L^2(\Gamma_-)$ are given scalar-valued functions, and $u_{\bm\beta}$ is the directional derivative of $u$ along the direction ${\bm\beta}$ defined by 
\begin{equation}\label{dd-d}
    u_{\bm\beta}(\bx) = D_{\bm{\beta}}u(\bx)= \lim_{\tau\to 0} \frac{u(\bx)-u\big(\bx - \tau\bm{\beta}(\bx)\big)}{\tau}.
\end{equation}
Note that \cref{pde1} holds for discontinuous solutions on discontinuity interfaces, while ${\bm \beta}\cdot\nabla u$ is invalid (see \cite{LiuCai2026} for details).

Denote the solution space by 
\[
V_{\bm\beta}= \left\{v\in L^2(\Omega): v_{\bm{\beta}}\in L^2(\Omega)\right\},
\]
equipped with the norm 
\[
 \|v\|_{\bm{\beta}}=\left(\|v\|_{0,\Omega}^2 + \|v_{\bm\beta} \|_{0,\Omega}^2\right)^{1/2}.
 \]
Denote the weighted $L^2(\Gamma_{-})$ norm on the inflow boundary by
 \[
 \|v\|_{0,\bm{\beta},\Gamma_-} 
 =\left( \int_{\Gamma_-} |\bm{\beta}\! \cdot \!\bm{n}|\, v^2\,ds\right)^{1/2}.
 \]
 Introducing the following least-squares functional
 \begin{equation}\label{ls1}
    \mathcal{L}(v;{\bf f}) = \|v_{\bm\beta} +\gamma\, v-f\|_{0,\Omega}^2 +  \|v-g\|_{0,\bm{\beta},\Gamma_-}^2, \quad \forall\,\, v\in V_{\bm\beta}
\end{equation}
with ${\bf f} = (f,g)$, then the least-squares formulation of problem \cref{pde1} is to seek $u\in V_{\bm\beta}$ such that
\begin{equation}\label{minimization1}
    \mathcal{L}(u;{\bf f}) = \min_{ v\in V_{\bm\beta}} \mathcal{L}(v;{\bf f}).
\end{equation}

When the solution is smooth, this formulation was studied in \cite{bochev2001improved, de2004least, bochev2016least}, and the coercivity of the homogeneous least-squares functional was established under the assumption that there exists a positive constant $\gamma_0>0$ such that
\[
\gamma (\bx) - \dfrac12 \nabla\cdot {\bm\beta} \ge \gamma_0, \quad\forall\,\, \bx\in \Omega.
\]

The least-squares ReLU neural network (LSNN) method was introduced in \cite{Cai2021linear} for solving the linear advection-reaction equation in \cref{pde1} with discontinuous solutions. The LSNN is one of the physics-preserving neural network (P$^2$NN) methods \cite{LiuCaiP2NN} and is specially designed for problems without natural minimization principle. Based on the $L^2$-norm least-squares formulation in \cref{ls1}, this method employs ReLU neural networks (NN) as approximating functions and physics-preserving numerical differentiation operator. Without using a priori knowledge of the discontinuity interface location, the LSNN is capable of approximating discontinuous solutions without oscillation and overshooting. The LSNN method was also developed for scalar nonlinear hyperbolic conservation laws (see \cite{Cai2022nonlinear}).

Despite the impressive approximation capabilities of NNs, the resulting discretization leads to a non-convex optimization problem in the NN parameters. The methods of gradient descent type are often slow and likely trap in a local minimum so that the resulting NN approximation is inaccurate (see, e.g., \cite{Cai2021linear}). To overcome this difficulty, recently we have investigated superlinear and/or second-order optimization methods like Gauss-Newton (GN), Newton, etc. For example, in \cite{CaiDokFalHer2024a, CaiDokFalHer2024b} we developed a damped block Newton (dBN) method for one-dimensional elliptic problems. The dBN produces more accurate approximation than 
the adaptive finite element method (see, e.g., \cite{MoNoSi:02, Ve:13}) in a comparable computational cost. For solving the least-squares (LS) problems using a shallow ReLU neural network in multi-dimension, we introduced the structure-guided Gauss-Newton (SgGN) method in \cite{SgGN}. 

In the development of the SgGN method, \cite{SgGN} overcame two essential difficulties: less differentiability of the ReLU function and singularity of the GN matrix. First, the ReLU activation function is not differentiable in a pointwise manner, hence the methods of gradient type are not applicable to the discrete least-squares problems whose loss functions are defined as finite summations. By realizing 
that the ReLU has first-order weak derivatives, this difficulty was overcome by developing the SgGN method for the continuous least-squares problems and then extended to the discrete least-squares problems. Second, the symmetric GN matrix is often singular. The well-known Levenberg-Marquardt (LM) method \cite{Levenberg, Marquardt} deals with this singularity by the regularization technique and ends with a non-GN search direction. The second difficulty was overcome by using the algebraic structure of the shallow ReLU NN to remove all singularities of the GN matrix at each iteration step.
The idea of removing singularities of the Hessian matrix was also used in \cite{CaiDokFalHer2024a, CaiDokFalHer2024b}.

The purpose of this paper is to extend the SgGN method from the LS approximation problem to the LSNN method for solving the two-dimensional linear advection-reaction equation (a first-order partial differential equation), when using shallow ReLU neural networks. This extension encounters the same difficulties as in the LS approximation problem. While the latter is the same, the former is more severe than before since the loss functional contains derivatives of the ReLU function. 
For one-dimensional elliptic differential equations, we were able to compute not only first- but also second-order derivatives of the loss functional on the nonlinear parameters and derived both the GN and the Hessian matrices exactly. This one-dimensional approach \cite{CaiDokFalHer2024a, CaiDokFalHer2024b} is difficult to be extended to multi-dimensions. In this paper, we will first establish the commutativity properties between the gradient on the nonlinear parameters and the integral involving the Heaviside step function (see \cref{commutativity}) and then derive the formula of the GN matrix. Again, the GN matrix has a similar algebraic structure as that of the LS approximation problem, and hence we can remove singularity of the GN matrix at each iteration step to obtain an efficient SgGN method. The SgGN is implemented, and its efficiency and accuracy are demonstrated numerically for all test problems in \cite{Cai2021linear}. In particular, the SgGN achieves much more accurate NN approximations than the Adam optimizer in much less computation.




The paper is structured as follows. \Cref{s:lsnn} introduces shallow ReLU neural networks, the discrete directional derivative, and the LSNN method. \Cref{s:Comm} establishes some commutativity properties between differentiation with respect to NN nonlinear parameters and integration of NN function with respect to independent variable $\bx$. \Cref{s:BNMethod} derives a SgGN method, covering optimality conditions, the construction of the Hessian and Gauss-Newton matrices, and some implementation details. Finally, numerical experiments demonstrating the performance of the proposed method are presented in \Cref{sec:exp}. 

\section{Least-Squares Neural Network Method}\label{s:lsnn}

This section describes the least-squares neural network (LSNN) method introduced in \cite{Cai2021linear}. For a comprehensive description of LSNN, see \cite{LiuCai2026}. Since this paper studies optimization schemes for shallow ReLU neural network (NN) approximation, we only describe LSNN using a shallow NN. 

To this end, for any $\bx = (x_1, x_2)^T \in \Omega$, denote by $\by = (1, x_1, x_2)^T$ the homogeneous coordinates. A shallow ReLU NN with $n$ neurons may be defined as the collection of continuous piecewise linear functions from $\Omega$ to $\mathbb{R}$: 
\begin{equation}\label{network}  
\cM_n(\Omega)=\left\{ c_0+\sum_{i=1}^nc_i\sigma(\br_i \cdot \by) :  c_i\in \R,\: \br_i = (b_i, \bomega_i)^T= (b_i, \omega_{i1}, \omega_{i2})^T, \: \bomega_i\in {\cal S}, \: b_i \in \R \right\},
\end{equation}
where $\sigma(t) = \max\{0, t\}$ is the rectified linear activation function (ReLU) and 
\[
{\cal S}=\left\{\bxi\in \R^2 :\, |\bxi|
=1  \right\}
\]
is the unit circle in $\R^2$. Here, $|\bxi|$ denotes the magnitude of $\bxi$. 
The bias and weights of the output and hidden layers, $\bc =\left(c_0,c_1,\ldots,c_{n}\right)^T$ and $\br  = (\br_1, \dots, \br_n)$, are referred to as the linear and nonlinear parameters, respectively.
The set $\cM_n(\Omega)$ has a striking approximation property for singular and/or discontinuous functions with unknown location. Particularly, it was shown in \cite{Cai2021linear} that a step function with an {\it unknown} straight line interface can be approximated by a function in $\cM_n(\Omega)$ with two neurons for any given accuracy (see \cite{CCL2024} for multi-dimension and hypersurface interface). 
 
Evaluation of the LS functional $\mathcal{L}(v;{\bf f})$ in \cref{ls1} involves differentiation and integration. In practice, the former, the directional derivative, may be approximated by 
\begin{equation}\label{finite_diff}
     v_{\bm\beta}(\bx) \approx D_{\scriptstyle{\bm\beta},\tau}\,v(\mathbf{x})  
     = \frac{v(\bx)-v\big(\bx - \tau {\bm{\beta}}(\bx)\big)}{\tau} 
\end{equation}
for $0<\tau \ll 1$. Equation \cref{finite_diff} defines the discrete directional derivative operator $D_{\scriptstyle{\bm\beta},\tau}$ and an upwind finite difference scheme of the directional derivative with step size $\tau$. For all integration point $\mathbf{x} \in \Omega$ adjacent to $\Gamma_-$, if the step size $\tau$ is properly chosen such that $\bx - \tau {\bm{\beta}}(\bx) \in \Gamma_-$, then \cref{finite_diff} can be used to enforce the inflow boundary condition by setting
\[
v(\bx - \tau {\bm{\beta}}(\bx)) = g_\tau(\mathbf{x}) := g(\bx - \tau {\bm{\beta}}(\bx)).
\]

This approach (see, e.g., \cite{LiuCai2026}) requires to approximate integral by numerical integration first and then to apply the discrete directional derivative at integration points. Note that the integrand of the LS functional involves the Heaviside step function. It is then difficult to compute the gradient of the LS functional evaluated at integration points with respect to the nonlinear parameters. 
To overcome this obstacle, we do not use numerical integration before deriving optimization scheme of the Gauss-Newton type. To this end, let $\tilde{\Omega}_B$ be a strip region of the domain $\Omega$ with the width comparable to the size of numerical integration mesh (see \cref{imp}) such that the inflow boundary lies on the boundary of $\tilde{\Omega}_B$, i.e., $\Gamma_- \subseteq \partial \tilde{\Omega}_B$. Let $ \Omega_I=\Omega\setminus \tilde{\Omega}_B$, then $\Omega = \Omega_I \cup \tilde{\Omega}_B$ is a partition of the domain $\Omega$.


First, the step size $\tau$ is chosen differently in $\Omega_I$ and $\tilde{\Omega}_B$. If $\mathbf{x} \in \Omega_I$, the $\tau$ is chosen as a small positive constant. For $\mathbf{x} \in \tilde{\Omega}_B$, $\tau = \tau(\mathbf{x})$ is selected such that the backward point $\mathbf{x} - \tau\bm{\beta}(\mathbf{x})$ lies on the inflow boundary $\Gamma_{-}$. Consequently, the inflow boundary condition in \cref{pde1} is enforced through the discrete directional derivative for $\bx \in \tilde{\Omega}_B$ as follows
\begin{equation}\label{dd-d2}
D_{{\bm \beta}, \tau } v(\bx) = \frac{v(\bx) - g_\tau(\bx)}{\tau(\bx)}.
\end{equation}
If $\mathbf{x}\in \tilde{\Omega}_B$ is close to the inflow boundary $\Gamma_-$, the $\tau(\mathbf{x})$ could be very small. To avoid unbounded $1/\tau$ for those points, we replace  the domain $\tilde{\Omega}_B$ by a subdomain ${\Omega}_{B} := \tilde{\Omega}_B \setminus \Omega^{\epsilon}_{B}$ (see \cref{fig:domain_decomp}), where $\Omega_B^{\epsilon} := \{\mathbf{x} \in \tilde{\Omega}_{B} : \text{dist}(\mathbf{x}, \Gamma_{-}) \leq \epsilon \}$ is a $\epsilon$ strip subdomain of $\tilde{\Omega}_B$ for a given small $\epsilon > 0$. Clearly, for all $\mathbf{x} \in {\Omega}_B$, we have
\[
\dfrac{1}{\tau(\mathbf{x})} \geq \dfrac{1}{\epsilon}.
\]
Hence, $D_{{\bm \beta}, \tau } v(\bx)$ is well defined in ${\Omega}_{B}$.

\begin{figure}[htbp]
    \centering
    \includegraphics[width=0.3\textwidth]{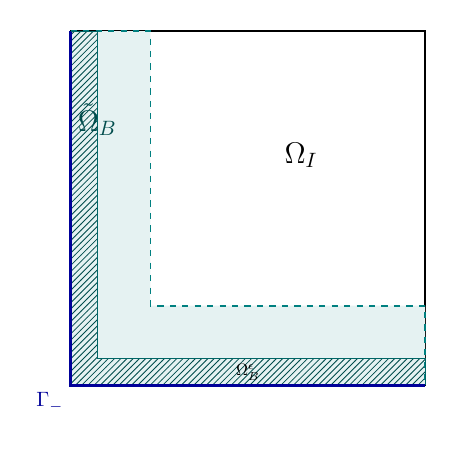}
    \caption{Decomposition of the domain $\Omega$.}
    \label{fig:domain_decomp}
\end{figure}

Next, we introduce a semi-discretized least-squares functional over the computational domain $\Omega_I\cup {\Omega}_B$ as follows
\begin{equation}\label{ls_tau}
    \mathcal{L}_{\tau}(v; {\bf f}) = \mathcal{L}^I_{\tau}(v; f) + \mathcal{L}^B_{\tau}(v; {\bf f}).
\end{equation}
where the functionals $\mathcal{L}^I_{\tau}(v; f)$ and $\mathcal{L}^B_{\tau}(v; {\bf f})$ are given by
\begin{align*}
    \mathcal{L}^I_{\tau}(v; f) = \dfrac12 \|D_{\bm{\beta}, \tau} v + \gamma v - f\|^2_{0, \Omega_I} \quad \text{and} \quad \mathcal{L}^B_{\tau}(v; {\bf f}) = \dfrac12 \|(\tau^{-1} + \gamma)v - \tau^{-1}g_\tau - f\|^2_{0, {\Omega}_B},
\end{align*}
respectively. Now, the LSNN method for problem \cref{pde1} seeks $u_{n, \tau} \in \mathcal{M}_n(\Omega)$ such that 
\begin{equation}\label{discrete_minimization_functional}
    \mathcal{L}_{\tau}(u_{n, \tau}; {\bf  f}) = \min_{v \in \mathcal{M}_n(\Omega)} \mathcal{L}_{\tau}(v; {\bf f}).
\end{equation}


\section{Commutativity}\label{s:Comm} 

To derive the structure-guided Gauss-Newton method, this section develops some commutativity properties between differentiation with respect to NN nonlinear parameters and integration of NN function with respect to spatial variable $\bx$. 

To this end, denote the Heaviside step and the Dirac delta functions by
\begin{equation}\label{H-delta}
    H(t) = \left\{\begin{array}{ll}
    0, & t<0,\\[2mm]
    1/2, & t=0,\\[2mm]
    1, & t> 0,
    \end{array}\right.
    \quad \mbox{ and }\quad 
    \delta(t)=\left\{\begin{array}{ll}
    \infty, & t=0,\\[2mm]    0, & t\neq 0, 
    \end{array}\right. \mbox{ such that } \int^{+\infty}_{-\infty}\delta(t)\,dt=1,
\end{equation}
respectively. It is well-known that the weak derivatives \cite{adams2003sobolev} of $\sigma(t)$ and $H(t)$ are the respective step and delta functions: 
\begin{equation}\label{derivative}
    H(t) = \sigma^\prime(t)
    \quad \mbox{ and }\quad 
    \delta(t) = H^\prime(t). 
\end{equation}

The $\delta(t)$ is a generalized function and may be defined as the limit of a family of step functions as
\begin{equation}\label{delta}
\delta(t)=\lim\limits_{\tau \to 0^+} d_\tau(t) \quad\mbox{with }\,\, d_\tau(t) = \left\{\begin{array}{ll}
        \dfrac{1}{2\tau}, & t \in (-\tau, \tau),\\[3mm]
        0, & \text{otherwise}.
        \end{array}\right.
\end{equation}
For a given continuous function $\varphi: \mathbb{R} \rightarrow \mathbb{R}$ with compact support and for $a, b \in \mathbb{R}$ with $b \neq 0$, the translation and scaling properties \cite{Dirac1930} are valid:  
\begin{equation}\label{deltaIntegration}
    \int_{-\infty}^{\infty}\varphi(t)\delta(t - a)\,dt = \varphi(a) \quad \mbox{and} \quad \int_{-\infty}^{\infty}\varphi(t)\delta(bt)\, dt = \frac{\varphi(0)}{|b|}.
\end{equation}

Let $\hat{b}\in\R$ and $\hat{\br}=(\hat{b}, \hat{\omega}_1,\hat{\omega}_2)^T$ with  
$\hat{\bomega}=(\hat{\omega}_1,\hat{\omega}_2)^T\in \mathcal{S}$, 
each neuron function $\sigma (\hat{\br}\cdot \by)$ in the domain $\Omega$ is a continuous piecewise linear function with one breaking line
\[
{\bf l}_{\hat\br}=\{\bx\in \Omega : \hat{\br}\cdot \by = \hat{b} + \hat{\bomega} \cdot \bx =0\}.
\]
Let $\hat{\ba}$ and $\hat{\bb}$ be the intersections between the line segment ${\bf l}_{\hat\br}$ and the boundary of $\Omega$ such that $\hat{\bomega}$ points to the right-hand side when walking from $\hat{\ba}$ to $\hat{\bb}$. Then a parametrization of the line segment ${\bf l}_{\hat\br}$ is given by
\begin{equation}\label{l}
  {\bf l}_{\hat\br}(t) 
  =\hat{\ba} + t\, (\hat{\bb}-\hat{\ba}) \quad\mbox{for }\, t\in [0, 1].  
\end{equation}
Denote by $\hat{\bomega}^\perp= (-\hat{\omega}_2,\hat{\omega}_1)$ the unit vector orthogonal to $\hat{\bomega}$, then $\hat{\bomega}^\perp$ is parallel to $\hat{\bb}-\hat{\ba}$, i.e., 
\[
\hat{\bb}-\hat{\ba} = \pm |\hat{\bb}-\hat{\ba}|\, \hat{\bomega}^\perp.
\]

\begin{lemma}\label{diracDeltaLemma}
    Let ${\bf l}_{\hat{\br}}$ and ${\bf l}_{\tilde{\br}}$ be two distinct line segments in ${\Omega}$. Assume that the line segment ${\bf l}_{\hat{\br}}$ is non-empty. Let $h(\bx)$ be a continuous or piecewise continuous function with respect to the partition of ${\Omega}$ by the line segment ${\bf l}_{\tilde{\br}}$, then
    \begin{equation}\label{h-d}
    \int_{{\Omega}}h(\bx)\delta(\hat{\br}\cdot \by)\,d\bx = \int_{{\bf l}_{\hat{\br}}} h\,ds, 
    \end{equation}
    where the right-hand side is the line integral of $h(\bx)$ along the line segment ${\bf l}_{\hat{\br}}$.
\end{lemma}

\begin{proof}
First, assume that $h(\bx)$ is continuous in ${\Omega}$. For small $\tau>0$, denote the $2\tau$-width strip centered at the line $\hat{\br}\cdot\by=0$ and a truncated rectangular strip of $R_{\tau}$  by 
\[
R_{\tau}= \left\{\bx \in \R^2 \big| -\tau \leq \hat{\bomega}\cdot\bx+\hat{b} \leq \tau\right\} \,\mbox{ and }\, R_{\tau}(a)= \left\{\bx \in R_\tau \big| -a \leq \hat{\bomega}^\perp\cdot (\bx- (\hat{\bb}-\hat{\ba})/2)\leq a\right\},
\]
respectively. 
Assume that $a$ is large enough so that $R_{\tau}\cap \Omega \subset R_{\tau}(a)$. Let $\tilde{\Omega}$ be a rectangular region containing $\Omega \cup R_{\tau}(a)$. Denote by $\tilde{h}(\bx)$ the extension of $h(\bx)$ to $\tilde{\Omega}$ by setting $\tilde{h}(\bx)=0$ in $\tilde{\Omega}\setminus \Omega$. 

Let $\xi=\hat{\bomega}\cdot\bx+\hat{b}$ and $\eta=\hat{\bomega}^\perp\cdot (\bx- (\hat{\bb}-\hat{\ba})/2)$. It follows from \cref{delta}, the definition of $\tilde{h}(\bx)$, Fubini's theorem, and the continuity of $h(\bx)$ that
\begin{eqnarray*}
\int_{\Omega}h(\bx)\delta(\hat{\br}\cdot \by)\,d\bx &=& \lim\limits_{\tau\to 0^+}\int_{\Omega}h(\bx) d_\tau(\hat{\br}\cdot \by)\,d\bx
= \lim\limits_{\tau\to 0^+} \int_{\tilde{\Omega}}\tilde{h}(\bx) d_\tau(\hat{\br}\cdot \by)\,d\bx \\[2mm]
&=& \lim\limits_{\tau\to 0^+}\dfrac{1}{2\tau}\int_{-\tau}^\tau\int_{-a}^{a}\tilde{h}(\bx(\xi, \eta))\,d\eta d\xi = \int_{-a}^{a}\tilde{h}(\bx(0, \eta))\,d\eta =\int_{{\bf l}_{\hat{\br}}} h\,ds. 
\end{eqnarray*}
This proves \cref{h-d} for the case of $h(\bx)$ being continuous. 
The discontinuous case follows by separately considering the subsets of $\Omega$ divided by ${\bf l}_{\hat{\br}}$ and summing the corresponding integrals.
\end{proof}


With \cref{diracDeltaLemma}, we may use \cref{h-d} to define the integral of a piecewise continuous function against the Dirac delta function in two dimensions as the line integral. Next lemma establishes the commutativity between the gradient and the integral involving the Heaviside step function.

\begin{lemma}\label{commutativity}
 For a given function $h(\bx)$ and a vector $\hat{\br}$ as in \cref{diracDeltaLemma}, under the assumptions of \cref{diracDeltaLemma}, we have
    \begin{align}\label{g-H}
    \nabla_{\hat{\br}} \int_{\Omega}h(\bx)H(\hat{\br}\cdot\by)\,d\bx &= \int_{\Omega}h(\bx)\delta(\hat{\br}\cdot\by)\by \,d\bx = \int_{\Omega}h(\bx) \nabla_{\hat{\br}}H(\hat{\br}\cdot\by)\,d\bx \\[2mm]  \label{hsigma} \mbox{and }\, \nabla_{\hat{\br}}\int_{\Omega}h(\bx)\sigma(\hat{\br}\cdot\by)H(\hat{\br}\cdot\by)d\bx &= \int_{\Omega}h(\bx)\sigma(\hat{\br}\cdot\by)\delta(\hat{\br}\cdot\by)\by d\bx + \int_{\Omega}h(\bx)\left[H(\hat{\br}\cdot\by)\right]^2\by d\bx.     
    \end{align}
    where $\nabla_{\hat{\br}}$ denotes the weak gradient with respect to $\hat{\br}$.
\end{lemma}

\begin{proof}
Because the proof of each equation in \cref{g-H} is similar, then it suffices to show the validity of \cref{g-H} for the first equation, i.e.,
\begin{equation}\label{g-H-b}
    \dfrac{\partial}{\partial \hat{b}} \int_{\Omega}h(\bx)H(\hat{\br}\cdot\by)\,d\bx = \int_{\Omega}h(\bx)\delta(\hat{\br}\cdot\by) \,d\bx,     
    \end{equation}
where $\frac{\partial}{\partial \hat{b}}$ denotes the weak derivative with respect to $\hat{b}$. To this end, for any $\varphi \in C^{\infty}_0(\R)$, it follows from Fubini's theorem, integration by parts, \cref{delta}, and the dominated convergence theorem that
    \begin{eqnarray*}
       && \int_{-\infty}^{\infty} \biggl(\int_{\Omega} h(\bx)H(\hat{\br}\cdot\by)d\bx \biggr)\varphi^\prime(\hat{b})\,d\hat{b} =\int_{\Omega}h(\bx)\biggl(\int_{-\infty}^{\infty}H(\hat{\br}\cdot\by)\varphi'(\hat{b})\,d\hat{b}\biggr)d\bx\\[2mm]
    &=&-\int_{\Omega}h(\bx)\biggl( \int_{-\infty}^{\infty} \delta(\hat{\br}\cdot\by)\varphi(\hat{b})\,d\hat{b} \biggr)d\bx
    =-\lim_{\tau \to 0} \int_{\Omega} h(\bx)\biggl(\int_{-\infty}^{\infty} d_\tau (\hat{\br}\cdot\by)\varphi(\hat{b})\,d\hat{b} \biggr)d\bx\\[2mm]
        & =& -\lim_{\tau \to 0}\int_{-\infty}^{\infty}\biggl(\int_{\Omega}h(\bx) d_\tau(\hat{\br}\cdot\by) d\bx\biggr)\varphi(\hat{b})\,d\hat{b} 
        = -\int_{-\infty}^{\infty}\biggl(\int_{\Omega}h(\bx) \delta (\hat{\br}\cdot\by) d\bx\biggr)\varphi(\hat{b})\,d\hat{b},       
    \end{eqnarray*}
which implies the first equality of \cref{g-H-b}. Together with $\nabla_{\hat{\br}}H(\hat{\br}\cdot\by)= \delta(\hat{\br}\cdot\by)\by$ yields the second equality of \cref{g-H-b}.

To show the validity of \cref{g-H}, we first notice that 
\[
\nabla_{\hat{\br}}\int_{\Omega}h(\bx)\sigma(\hat{\br}\cdot\by)\, d\bx = \int_{\Omega}h(\bx) \nabla_{\hat{\br}}\sigma(\hat{\br}\cdot\by)\, d\bx = \int_{\Omega}h(\bx) H(\hat{\br}\cdot\by)  \by\, d\bx.
\]
\cref{hsigma} is then a direct consequence of \cref{diracDeltaLemma} and the facts that
\[
\sigma( \hat{\br} \cdot\by) H(\hat{\br}\cdot\by) = \sigma(\hat{\br} \cdot\by), \quad H(\hat{\br}\cdot\by)^2 = H(\hat{\br}\cdot\by), \quad\mbox{and}\quad \sigma(\hat{\br} \cdot \by)\big|_{{\bf l}_{\hat{\br}}} = 0.
\]
This completes the proof of the lemma.
\end{proof}

\begin{definition}\label{intH_id_i}
    For a given function $h(\bx)$ and a vector $\hat{\br}$ as in \cref{diracDeltaLemma}, we define
    \begin{equation}\label{HD}
        \int_{\Omega}h(\bx)H(\hat{\br} \cdot \by)\delta(\hat{\br} \cdot \by )\,d\bx = \frac{1}{2}\int_{\Omega}h(\bx)\delta(\hat{\br} \cdot \by)\, d\bx.
    \end{equation}
\end{definition}

\section{A Structure-Guided Gauss-Newton (SgGN) Method}\label{s:BNMethod}Since  the functional $\mathcal{L}_{\tau}\big(v;{\bf f}\big)$ in \cref{discrete_minimization_functional} is a non-convex function of the NN parameters, the resulting discrete optimization problem is non-convex. Non-convex optimization problem is generally computationally intensive and complicated. This section discusses the algebraic structures of this optimization problem and how to design a SgGN method for solving the minimization problem in \cref{discrete_minimization_functional}.

\subsection{Optimality Conditions}We first derive systems of algebraic equations arising from the first order optimality conditions of the least-squares problem in \cref{discrete_minimization_functional}. 
A solution to problem \cref{discrete_minimization_functional} has the form of
\begin{equation}\label{nn_sum}
u_{n, \tau}(\bx) = \sum_{i=0}^{n} c_i \sigma_{i}(\bx) = \bSigma(\bx;\br)^T\bc,
\end{equation}
where $\bSigma(\bx;\br)= \left(\sigma_0(\bx),\ldots, \sigma_n(\bx;\br)\right)^T$ and 
\[
\sigma_{0}(\bx)=1
\quad\mbox{and}\quad 
\sigma_{i}(\bx;\br)= \sigma  \left(b_i +\bomega_i\cdot\bx \right) = \sigma  \left(\br_i\cdot\by \right) .
\]

The linear and the nonlinear parameters $\bc$ and $\br$ 
satisfy the following optimality conditions
\begin{equation}\label{OC}
    \nabla_{\bc} {\cal L}_{\tau}\big(u_{n, \tau};{\bf f}\big)  ={\bf 0}
    \quad\mbox{and}\quad    \nabla_{\br}{\cal L}_{\tau} \big(u_{n, \tau};{\bf f}\big)  ={\bf 0},
\end{equation}
respectively, where $\nabla_{\bc}$ and $\nabla_{\br}$ denote the gradients with respect to the respective $\bc$ and $\br$.

First, we compute $\nabla_{\bc} {\cal L}_{\tau}\big(u_{n, \tau};{\bf f}\big)$. By \cref{ls_tau} and \cref{nn_sum}, we have 
\begin{equation*}
    {\cal L}_{\tau}\big(u_{n, \tau};{\bf f}\big) = \dfrac{1}{2}\|(D_{{\bm \beta, \tau}}+\gamma)\bSigma^T\bc - f\|_{\Omega_I}^2 + \dfrac{1}{2} \|(\tau^{-1} + \gamma)\bSigma^T\bc -\tau^{-1}g_{\tau}- f\|_{{\Omega}_B}^2,
\end{equation*}
where $D_{{\bm \beta, \tau}}+\gamma$ is applied componentwisely. Since $\mathcal{L}_{\tau} \big(u_{n, \tau};{\bf f}\big)$ is quadratic with respect to the linear parameters $\bc$, it is then easy to check that 
\begin{equation}\label{linear system2}
   {\bf 0}=\nabla_{\bc} {\cal L}_{\tau}\big(u_{n, \tau};{\bf f}\big) = \bm{A}(\br) \bc - F(\br),
\end{equation}
where $\bm{A}(\br)$ and $F(\br)$ are the stiffness matrix and the right-hand side vector given by 
\begin{equation}\label{A-F}
\left\{\begin{array}{l}
   \bm{A}\left(\br\right)=\displaystyle\int_{\Omega_I}\left(\left[(D_{\bm{\beta}, \tau}+\gamma) \bSigma\right] \left[(D_{\bm{\beta}, \tau} +\gamma) \bSigma\right]^T\right) d\bx + \int_{\Omega_B}\left([\tau^{-1} + \gamma]^2\bSigma\bSigma^T\right) d\bx \\ [6mm] 
\mbox{and}\quad F\left(\br\right)=\displaystyle\int_{\Omega_I}\left(f \left[(D_{\bm{\beta}, \tau} +\gamma) \bSigma\right] \right)d\bx + \int_{\Omega_B} \left(\left[\tau^{-1}+\gamma\right]\left[\frac{g_\tau}{\tau}+f\right]\bSigma\right)d\bx,
\end{array}\right.
\end{equation}
respectively. For each $i = 0, 1, \dots, n$, let 
\[
\varphi_i(\bx) = \varphi_i(\bx; \br) =
\left\{\begin{array}{ll}
    D_{\bm\beta, \tau}\sigma_i(\bx) + \gamma\sigma_i(\bx), & \bx \in \Omega_I, \\[2mm] 
     \left(\tau^{-1}(\bx) +  \gamma\right)\sigma_i(\bx), & \bx \in \Omega_B.
    \end{array}\right. 
\]

\begin{lemma}\label{lem:phi_i}
Under the assumptions that the breaking lines are distinct {\em (}i.e., $\bl_{\br_i} \neq \bl_{\br_j}$ for all $i \neq j${\em )} and that their intersection with $\Omega_{I} \cup {\Omega}_B$ is non-empty {\em (}$\bl_{\br_i} \cap \left(\Omega_{I}\cup{\Omega}_B\right) \neq \emptyset$ for all $i${\em )}, the set $\{\varphi_i(\bx)\}_{i=0}^n$ is linearly independent.
\end{lemma}

        \begin{proof}For any $\alpha_i \in \mathbb{R}$ ($i=0,1,\ldots,n$),  let $\Phi(\bx) = \sum\limits_{i = 0}^{n}\alpha_i \varphi_i(\bx)$. Additionally, for each $i = 1, \dots, n$, we define the breaking curve $\mathcal{C}_{\br_i}$ by 
         $${\cal C}_{\br_i} = \{\bx \in \Omega: b_i + \bomega_i \cdot (\bx - \tau {\bm \beta}(\bx)) = 0\}.$$  Introducing the following index sets
\begin{align*}
    I_{1, I} =& \{i \in\{1,\dots, n\}: D_{{\bm \beta}, \tau}\sigma_i \neq 0, (\bl_{\br_i} \cup \bl_{\br_i, \tau}) \cap \Omega_I \neq \emptyset \},\\[2mm]
    I_{1, B} = &\{i \in\{1,\dots, n\}: D_{{\bm \beta}, \tau}\sigma_i \neq 0, (\bl_{\br_i} \cup \bl_{\br_i, \tau}) \cap \Omega_I = \emptyset \} \cup \{0\},\\[2mm]
    \mbox{and} \,\, I_2 =& \{1,\dots, n\}\backslash (I_{1,I} \cup I_{1,B}),
\end{align*}
then, for any $\bx \in \Omega_I$ we have
\[
\Phi(\bx) = \sum_{i \in I_{1, I}}\alpha_{i}\varphi_i(\bx) + \sum_{i \in I_{1, B}}\alpha_{i}\varphi_i(\bx) + \sum_{i \in I_2}\alpha_i\varphi_{i}(\bx).
\]

First, for all $i\in I_{1, B}$, $\varphi_i(\bx)$ is a linear function in $\Omega_I$, and so is $\sum\limits_{i \in I_{1, B}}\alpha_{i}\varphi_i(\bx)$. Second, for all $i\in I_{2}$, by \cref{finite_diff}, we have $\sigma_i(\bx) = \sigma_i(\bx-\tau {\bm \beta})$ for all $\bx \in \Omega_I$. If ${\bm \beta}$ is not constant, then $I_2$ is empty. Otherwise, the $\mathcal{C}_{\br_i}$ is a line parallel to ${\bm \beta}$. The assumption that $\bl_{\br_i} \cap (\Omega_I \cup \Omega_B) \neq \emptyset$ implies that $\bl_{\br_i} \cap \Omega_S \neq \emptyset$ for both $S=I$ and $S=B$. Therefore, for $i \in I_{1, I} \cup I_2$, $\varphi_i$ is a continuous piecewise linear function in $\Omega_I$. 

If $\Phi(\bx) = 0$ for all $\bx \in \Omega$, then it is also vanished in $\Omega_I\subset \Omega$. Hence, for  all $\bx \in \Omega_I$, we have
\[ 
    \sum_{i \in I_{1, I}}\alpha_{i}\varphi_i(\bx) + \sum_{i \in I_2}\alpha_i\varphi_{i}(\bx)=0, 
\] 
which, together with the fact that the breaking lines and curves are distinct, implies $\alpha_i = 0$ for all $i  \in I_{1, I} \cup I_2$. Now, the assumption $\Phi(\bx) = 0$ for all $\bx \in \Omega$ yields
$$\sum_{i \in I_{1, B}}\alpha_{i}\varphi_i(\bx) =0 \,\,\mbox{ in }\, \Omega_B.$$
Together with the fact that $\{\bl_{\br_i}\}_{i\in I_{1, B}}$ are distinct, we have $\alpha_i=0$ for all $i\in I_{1, B}$.  This completes the proof of the lemma.
\end{proof}

\begin{lemma}\label{lem:A_spd}
Under the assumptions of \cref{lem:phi_i}, the stiffness matrix $\bm{A}(\br)$ is symmetric positive definite.
\end{lemma}

\begin{proof}
This is a direct consequence of \cref{A-F} and \cref{lem:phi_i}.
\end{proof}

Next, we calculate 
\[
\nabla_{\br} \mathcal{L}_{\tau}\big(u_{n, \tau};{\bf f}\big)=\nabla_{\br} \mathcal{L}^I_{\tau}\big(u_{n, \tau};f\big)+\nabla_{\br} \mathcal{L}^B_{\tau}\big(u_{n, \tau};{\bf f}\big).
\]
To this end, denote by
\begin{equation}\label{R1}
R_S(\bx) = R_S(\bx;\bc,\br) =
\left\{\begin{array}{ll}
    \left[(D_{\bm\beta, \tau} + \gamma) u_{n, \tau} - f\right](\bx), & S=I, \\[2mm] 
    \left[ \left(\tau^{-1} +  \gamma\right)  u_{n, \tau}  -\tau^{-1}g_{\tau} -f\right](\bx), & S=B
    \end{array}\right.
\end{equation}
the residuals in $\Omega_I$ and ${\Omega}_B$, respectively. 
Let $\delta_{ij}$ denote the Kronecker delta symbol and let $H_i(\bx) = H_i(\bx; \br) = H(b_i + \bomega_i \cdot \bx)$. Since
\[
\nabla_{\br_i}\sigma_j(\bx) = \delta_{ij}H_j(\bx)\,\by \quad \mbox{and} \quad  \nabla_{\br_i}D_{{\bm \beta}, \tau }\sigma_j(\bx) = \delta_{ij}D_{\bm\beta, \tau}(H_j(\bx)\,\by),
\]
then \cref{nn_sum} gives
\begin{equation}\label{R2}
\nabla_{\br_i}R_S(\bx;\bc,\br)=
\left\{\begin{array}{ll}
 c_i(D_{\bm\beta, \tau} +  \gamma) (H_i(\bx)\,\by), & S=I, \\[2mm] 
 c_i \left(\tau^{-1} +\gamma\right)H_i(\bx)\, \by, & S=B .
    \end{array}\right.
\end{equation}
Let $\hat{\bc} = (c_1, \dots, c_n)^T$, and denote by $D(\hat{\bc}) = \text{diag}(c_1, \dots, c_n)$ the diagonal matrix with $\hat{\bc}$ on the diagonal. Set
\[
{\bH}(\bx;\br)=\big(H_1\left(\bx;\br\right), \ldots, H_n\left(\bx;\br\right)\big)^T, 
\]
then \cref{R2} implies 
\begin{equation}\label{R3}
\nabla_{\br}R_S(\bx;\bc,\br)=
 \left(D(\hat{\bc})\otimes I_{3}\right)\mathbf{G}_S(\bx;\br),
\end{equation}
for $S=I$ and $B$, where $\otimes$ is the Kronecker product and 
\begin{equation}\label{G}
\mathbf{G}_I(\bx;\br)=\left(D_{\bm\beta, \tau} + \gamma\right)\bigl(\bH(\bx;\br) \otimes \by\bigr) \quad\mbox{and}\quad \mathbf{G}_B(\bx;\br) = \left(\tau^{-1} + \gamma\right)\bH(\bx;\br) \otimes \by. 
\end{equation}

\begin{lemma}\label{l:Grad}The  gradient of $\mathcal{L}_{\tau}\big(u_{n, \tau};{\bf f}\big)$ with respect to the nonlinear parameters $\br$ is given by
\begin{equation}\label{GradL}
   \nabla_\br \mathcal{L}_{\tau}\big(u_{n, \tau};{\bf f}\big) 
= \left(D(\hat{\bc})\otimes I_{3}\right) \sum_{S=I, B} \int_{\Omega_S} R_S(\bx;\bc,\br) \mathbf{G}_S(\bx;\br) d\bx.
\end{equation}
\end{lemma}

\begin{proof}
It follows from \cref{R1} and the commutativity of the gradient and integration that 
\[
  \nabla_{\br}\mathcal{L}_{\tau}\big(u_{n, \tau};{\bf f}\big) 
  = \dfrac12 \nabla_{\br} \sum_{S=I, B} \int_{\Omega_S} \left[R_S(\bx; \bc, \br)\right]^2 d\bx = \sum_{S=I, B} \int_{\Omega_S} R_S(\bx; \bc, \br)\nabla_{\br}R_S(\bx; \bc, \br) d\bx,\]
which, together with \cref{R3}, implies \cref{GradL}.
\end{proof}

\subsection{Newton and Gauss-Newton Matrices}
For given linear parameters $\bc$, the second equation in \cref{OC} is the system of nonlinear algebraic equations on $\br$ to be solved by the Gauss-Newton method. In this section, we derive the Gauss-Newton and Hessian matrices of the functional ${\cal L}_{\tau}\big(u_{n, \tau};{\bf f}\big)$ with respect to the nonlinear parameters $\br$. 

To this end, let
\[
\Lambda(\bx;\br)=\text{diag}(\delta_1(\bx;\br_1),\ldots, \delta_n(\bx;\br_n)) \quad\mbox{with }\, \delta_i(\bx) = \delta_i(\bx;\br_i) = \delta (\br_i \cdot\by).
\]
For $i,j=1,\ldots,n$, we have $\nabla_{\br_i} \left(H_j(\bx)\,\by\right)^T = \delta_{ij}\delta_j(\bx)\,\by\by^T$. Hence,
\[
    \nabla_{\br} \bigl(\bH(\bx;\br) \otimes \by\bigr)^T =\Lambda(\bx;\br) \otimes  \by\by^T. 
\]
which, together with \cref{G}, implies,
\begin{equation}\label{nabla_rG}
\nabla_{\br}\mathbf{G}_S(\bx;\br)^T =\left\{\begin{array}{ll}
 (D_{\bm\beta, \tau}+\gamma)\bigl(\Lambda(\bx;\br) \otimes \by\by^T\bigr),    & S=I, \\[2mm]
   (\tau^{-1} + \gamma) (\Lambda(\bx; \br) \otimes \by \by^T),   & S=B.
\end{array}\right. 
\end{equation}

\begin{lemma}
Assume that $f$, $\gamma$, and $g$ are piecewise continuous. Then the Hessian matrix of $\mathcal{L}_{\tau}\big(u_{n, \tau};\bff\big)$ with respect to the nonlinear parameters $\br$ is given by
\begin{equation}\label{hessian}
    \nabla^2_\br \mathcal{L}_{\tau}\big(u_{n, \tau};\bff\big) =  \left(D(\hat{\bc})\otimes I_{3}\right){\cal H}(\br)\left(D(\hat{\bc})\otimes I_{3}\right) + \,\hat{{\cal H}}(\bc, \br)\left(D(\hat{\bc})\otimes I_{3}\right),
\end{equation}
where ${\mathcal{H}}(\br)= \sum\limits_{S=I, B} {\mathcal{H}}_S(\br)$ and $\hat{{\cal H}}(\bc, \br) = \sum\limits_{S=I, B} \hat{{\cal H}}_S(\bc, \br)$, and ${\mathcal{H}}_S(\br)$ and $\hat{{\cal H}}_S(\bc, \br)$ are given by
\begin{equation}\label{H}
{\mathcal{H}}_S(\br)= \int_{\Omega_S}  \mathbf{G}_S(\bx;\br) \mathbf{G}_S(\bx;\br)^Td\bx
\quad\mbox{and}\quad  \hat{{\cal H}}_S(\bc, \br)= \displaystyle\int_{\Omega_S}R_S(\bx; \bc, \br)\nabla_\br {\bf G}_S(\bx; \br)^T \,d\bx.
\end{equation}
\end{lemma}

\begin{proof}
Since $R_S(\bx; \bc, \br)$ and $\mathbf{G}_S(\bx;\br)$ defined in \cref{R1} and \cref{G} are functions of the respective $\left\{ \sigma_i(\bx)\right\}_{i=0}^n$ and  $\left\{ H_i(\bx)\right\}_{i=0}^n$, it follows from \cref{commutativity}, the product rule, \cref{R3}, and \cref{H} that
\begin{eqnarray*}
&& \nabla_\br \int_{\Omega_S} R_S(\bx;\bc,\br) \mathbf{G}_S(\bx;\br)^T d\bx  \\[2mm]
&=& \int_{\Omega_S} \left[\nabla_\br R_S(\bx;\bc,\br)\right] \mathbf{G}_S(\bx;\br)^T d\bx + \int_{\Omega_S} R_S(\bx;\bc,\br) \nabla_\br\mathbf{G}_S(\bx;\br)^T d\bx \\[2mm]
&=&  \left(D(\hat{\bc})\otimes I_{3}\right) {\mathcal{H}}_S(\br) + \hat{{\cal H}}_S(\bc, \br)
\end{eqnarray*}
which, together with \cref{GradL}, implies 
\[
\nabla^2_\br \mathcal{L}_{\tau}\big(u_{n, \tau};\bff\big) 
    =  \sum_{S=I, B} \left( \left(D(\hat{\bc})\otimes I_{3}\right) {\mathcal{H}}_S(\br) + \hat{{\cal H}}_S(\bc, \br) \right) \, \left(D(\hat{\bc})\otimes I_{3}\right).
\]
This completes the proof of the lemma.
\end{proof}

The first term of \cref{hessian} is the so-called Gauss-Newton matrix denoted by
\begin{equation}\label{GNMatrix}
    {\cal G}(\bc, \br) = \left(D(\hat{\bc})\otimes I_{3}\right){\cal H}({\br})\left(D(\hat{\bc})\otimes I_{3}\right),
\end{equation}
and ${\cal H}({\br})$ is referred to as the layer Gauss-Newton matrix. For $i = 1, \ldots, n$, let $$\psi_i(\bx) = \psi_i(\bx; \br) =
\left\{\begin{array}{ll}
    D_{\bm\beta, \tau}H_i(\bx) + \gamma H_i(\bx), & \bx \in \Omega_I, \\[2mm] 
     \left(\tau^{-1}(\bx) +  \gamma\right)H_i(\bx), & \bx \in \Omega_B.
    \end{array}\right.$$

\begin{lemma}\label{lem:phi2_i}
Under the assumptions of \cref{lem:phi_i}, the set  $\{\psi_i(\bx), x_1\psi_i(\bx), x_2\psi_i(\bx)\}_{i=1}^n$ is linearly independent.
\end{lemma}

\begin{proof}
For $i = 1, \ldots, n$, the linear independence of the set $\{1, x_1, x_2\}$ implies that the set $\{\psi_i(\bx), x_1\psi_i(\bx), x_2\psi_i(\bx)\}$ is linearly independent. Under the given assumptions, the linear independence of the collection $\{\psi_i(\bx), x_1\psi_i(\bx), x_2\psi_i(\bx)\}_{i=1}^n$ may be established in a similar fashion as the proof of \cref{lem:phi_i}.
\end{proof}

\begin{theorem}\label{GNspd}
For all $i, j = 1, \dots, n$, assume that $\bl_{\br_i} \cap (\Omega_{I} \cup {\Omega}_B) \neq \emptyset$ and $\bl_{\br_i} \neq \bl_{\br_j}$ if $i\neq j$. Then the layer Gauss-Newton matrix ${\cal H}(\br)$ is symmetric positive definite. Consequently, the Gauss-Newton matrix ${\cal G}(\bc, \br)$ is symmetric positive definite if and only if $c_i \neq 0$ for all $i = 1, \dots, n$.
\end{theorem}

\begin{proof}
Clearly, the matrices ${\cal H}(\br)$ and ${\cal G}(\bc, \br)$ are symmetric and positive semi-definite. To prove the validity of the theorem, by \cref{G} and \cref{H}, we have 
\begin{equation}\label{H(r)}
    {\cal H}(\br) = \displaystyle\int_{\Omega_I}\left(\left[(D_{\bm{\beta}, \tau}+\gamma)( \bH \otimes \by)\right] \left[(D_{\bm{\beta}, \tau}+\gamma)( \bH \otimes \by)\right]^T\right) d\bx + \int_{\Omega_B}\left(\left[\bH \otimes \by\right] \left[\bH \otimes \by\right]^T\right) d\bx.
\end{equation}
Consequently, the stated assumptions combined with \cref{lem:phi2_i} imply positive definiteness of 
${\cal H}(\br)$, which concludes the proof.
\end{proof}

\subsection{SgGN Method}\label{sec:method}
This section introduces a SgGN method for solving the systems of nonlinear algebraic equations in \cref{OC}, i.e., the optimality conditions of the non-convex optimization problem in \cref{discrete_minimization_functional}. Since this is a separable nonlinear least-squares problem (see, e.g., 
\cite{Golub1973, separable2, liu2008separable}), it is then natural to iterate back and forth between the linear (output) and the nonlinear (hidden layer) parameters, i.e., the block nonlinear Gauss-Seidel method as the outer iteration. The blocks are corresponding to the linear and the nonlinear parameters and are solved by a direct linear solver and a modified Gauss-Newton (GN) method per each outer iteration, respectively. 

A common issue for the GN method is that the GN matrix is generally symmetric, semi-positive. The most popular approach dealing with this issue is the Levenberg-Marquardt (LM) method \cite{Levenberg, Marquardt} through a regularization. Even though the regularized GN matrix is non-singular, the true GN search direction was changed. In this case, the LM method often performs like a type of methods of gradient descent (see, e.g., \cite{SgGN}). 



By \cref{GNspd}, the GN matrix ${\cal G}(\bc, \br)$ in \cref{GNMatrix} becomes singular only when $c_i = 0$ for some $i = 1, \dots, n$. As in \cite{SgGN, CaiDokFalHer2024a, CaiDokFalHer2024b}, if $c_i = 0$, then the $i^{th}$ neuron has no contribution to the current approximation 
and hence the associated nonlinear parameter $\br_i$ is not updated. Specifically, let 
\begin{equation}\label{setI}
    {\cal I} = \bigl\{i \in \{1, \dots, n\}:|c_i^{(k)}| \geq \epsilon_c \bigr\}
\end{equation}
be the set of contributing neurons, where $\epsilon_c > 0$ is a given tolerance. Denote by 
\[
{\cal I}^c=\{1,\ldots,n\}\setminus {\cal I}\]
the complement set of ${\cal I}$ and by $\left\lvert {\cal I}^c\right\rvert$ the number of elements in ${\cal I}^c$. For a given vector ${\bf v} \in \mathbb{R}^{3n}$, denote by ${\bf v}_{{\cal I}} \in \mathbb{R}^{3(n-\left\lvert {\cal I}^c\right\rvert)}$ as the vector obtained by removing the $i^{th}$, $(i+1)^{th}$, and $(i+2)^{th}$ entries from ${\bf v}$ for every $i \in {\cal I}^c=\{1,\dots, n\} \backslash {\cal I}$. Similarly, for a matrix ${\bf B} \in \mathbb{R}^{3n\times 3n}$, let ${\bf B}_{{\cal I}} \in \mathbb{R}^{3(n-\lvert {\cal I}^c\rvert)\times 3(n-\lvert {\cal I}^c\rvert)}$ be the matrix obtained by removing the $i^{th}$, $(i+1)^{th}$, and $(i+2)^{th}$ rows and columns of ${\bf B}$ for each $i \in {\cal I}^c$.
Then the \textit{reduced} Gauss-Newton direction vector is given by
\begin{equation}\label{reduced_GN}
    {\bf d}_{R}(\bc, \br) = 
    \left(D(\hat{\bc})\otimes I_{3}\right)_{\cal I}^{-1}{\cal H}_{\cal I}^{-1}({\br})\left(D(\hat{\bc})\otimes I_{3}\right)_{\cal I}^{-1}\Bigl(\nabla_{\br}\mathcal{L}_{\tau}\big(u_{n, \tau}(\bc, \br);{\bff}\big)\Bigr)_{\cal I}.
\end{equation}

We are now ready to describe a SgGN method (see \cref{alg:dBN} for a pseudocode) for solving the minimization problem in \cref{discrete_minimization_functional}. For a prescribed tolerance $\epsilon_c >0$, let $\br^{(k)}$ be the previous iterate, then the current iterate $(\bc^{({k})}, \br^{(k+1)})$  is computed as follows:
\begin{itemize}
    \item[(i)] \textit{Compute the linear parameters} 
    \begin{equation*}
    \bc^{(k)} = {\bm A}\left(\br^{(k)}\right)^{-1} F\left(\br^{(k)}\right).
    \end{equation*}
    \item[(ii)] \textit{Compute the search direction $\mathbf{p}^{(k)}=\left({\bf p}_1^{(k)},\dots,{\bf p}_n^{(k)}\right)^T \in \mathbb{R}^{3n}$}  
    \begin{equation}\label{directionVector}
    \bigl({\bf p}^{(k)}_i \bigr)_{i \in {\cal I}} = {\bf d}_R(\bc^{(k)}, \br^{(k)})  \quad \mbox{and} \quad \bigl({\bf p}^{(k)}_i\bigr)_{i \notin {\cal I}} = {\bf 0},
    \end{equation}
    where ${\bf d}_R(\bc, \br)$ is the \textit{reduced} Gauss-Newton direction vector defined in \cref{reduced_GN}.
    \item [(iii)] \textit{Compute the non-linear parameters}
    \begin{equation*}
        \br^{(k+1)} = \br^{(k)} +\mathbf{p}^{(k)}.
    \end{equation*}  
\end{itemize}

\begin{algorithm}
    \caption{A structure-guided Gauss-Newton (SgGN) method for  \cref{discrete_minimization_functional}}\label{alg:dBN} 
    \begin{algorithmic}
    \REQUIRE{Initial network parameters $\br^{(0)}$}
    \ENSURE{Network parameters $\bc$, $\br$}
    \FOR{$k=0,1\ldots$}
    \STATE $\triangleright$ \textit{Linear parameters}
    \STATE{$\bc^{(k)}\leftarrow {\bm A}\left(\br^{(k)}\right)^{-1} F\left(\br^{(k)}\right)$} 
    \STATE $\triangleright$
    \textit{Non-linear parameters}
    \STATE{{\textit{Compute the search direction} }$\mathbf{p}^{(k)}$ \textit{as in}} \cref{directionVector}
    \STATE{$\br^{(k+1)} \leftarrow \br^{(k)} + \mathbf{p}^{(k)}$}
\ENDFOR
    \end{algorithmic}
\end{algorithm}

\subsection{
Practical Considerations}\label{imp}
As discussed in \cite{SgGN}, implementation of the SgGN method has three practical issues: (1) inversion of the stiffness and Gauss-Newton matrices, (2) initialization, and (3) numerical integration.

First, per each iteration, the SgGN method requires inversions of the stiffness matrix ${\bm A}(\br)$ given in \cref{linear system2} and the reduced layer GN matrix ${\cal H}_{\cal I}(\br)$ in \cref{reduced_GN}. Both the matrices are symmetric positive definite (see \cref{lem:A_spd} and \cref{GNspd}), but they are ill-conditioned (see \cite{SgGN, CaiDokFalHer2024a, CaiDokFalHer2024b}). Therefore, the direct inversions are done through truncated SVDs for the experiments presented in \cref{sec:exp}.


Second, the minimization problem \cref{discrete_minimization_functional} is non-convex on the nonlinear parameters $\br$. Hence, it is crucial to start with a good initialization. Since $\sigma(\br_i \cdot \by)$ is continuous piecewise linear with the associated breaking line $l(\br_i)$, a natural choice is then to initialize the nonlinear parameters so that the resulting breaking lines form a uniform partition of the domain $\Omega$ (see, e.g., \cite{LiuCai1, Cai2021linear}).


When the solution of problem \cref{pde1} is discontinuous, the interface location is often determined by the characteristic curves emanating from discontinuities of the boundary data $g$. In such cases, the initialization could be set to align the corresponding breaking lines close to this interface. However, this strategy is not feasible for nonlinear hyperbolic conservation laws, where the location of the interface is unknown \textit{a priori}. For this reason, we evaluate our method using a uniform partition as the initial choice for the nonlinear parameters, aiming to demonstrate robustness where interface locations are not pre-determined.


Third, the SgGN method requires evaluation of integrals in the definitions of ${\bm A}(\br)$, $F(\br)$, ${\cal H}(\br)$, and $\nabla_{\br}{\cal L}_{\tau}$. For simplicity, the composite midpoint rule on a reasonably fine uniform partition of the domain is used in our numerical experiments (see \cref{sec:exp}). To reduce computational cost without compromising accuracy, a better choice is to use adaptive quadrature discussed in \cite{LiCaRa23, LiuCai1}.

\section{Numerical Experiments}\label{sec:exp}

This section presents numerical results for evaluating the performance of the SgGN method. In order to compare with Adam optimizer \cite{kingma2015}, a highly favored method of gradient descent type, done in \cite{Cai2021linear}, all test problems are limited to constant advection field ${\bm \beta}$. In such case, a ReLU neural network with one hidden layer is sufficient to accurately approximate discontinuous solutions with constant jumps across straight-line interfaces (see \cite{Cai2021linear}).

In all experiments, the composite midpoint rule is used for approximating integrals on the uniform partition of the computational domain $\Omega$ into squares of the size $h = 10^{-2}$.
The partition $\Omega = \Omega_I \cup \tilde{\Omega}_B$ is defined as follows: $\tilde{\Omega}_B$ is the union of the squares in the integration mesh that share an edge with the inflow boundary $\Gamma_-$, and $\Omega_I$ is the union of the remaining squares. The parameter $\tau$ in \cref{discrete_minimization_functional} is set to $10^{-5}$, and the tolerance $\epsilon_c$ in \cref{setI} is set to $10^{-8}$. Unless otherwise specified, in all test problems, the initial nonlinear parameters $\br^{(0)}$ are chosen so that the corresponding breaking lines form a uniform partition of the domain $\Omega$.

\subsection{Piecewise Constant Solution}
This section presents two test problems whose solutions are piecewise constants with the respective vertical and diagonal straight-line discontinuity interfaces. By Theorem 3.3 of \cite{Cai2021linear}, both the problems can be approximated by a one hidden layer ReLU neural network with two neurons ($n=2$) for any given accuracy $\epsilon >0$. 
The purpose of these test problems is to study dependence on initialization and performance of the SgGN method in terms of the number of iterations. 

\subsubsection{Vertical Line Interface}\label{test_P1}The first test problem is the equation in \cref{pde1} with the domain $\Omega = (0, 2) \times (0, 1)$, the inflow boundary $\Gamma_{-} = \{(x, 0) :x \in (0, 2)\}$, the constant advection field ${\bm \beta} = (0, 1)^T$,  $\gamma = f=0 $, and the inflow boundary data $g(x) = 0$ for $x \in (0, \pi/3)$ and $g(x) = 1$ for $x\in (\pi/3, 2)$. The exact solution for this problem is the piecewise constant function given by
\begin{equation*}
                u(x, y)   = 
                \begin{cases}
                  0, &  (x, y) \in\left(0,\pi/3 \right)\times\left(0, 2\right) ,  \\[2mm]
                  1, & (x, y) \in \left(\pi/3, 1 \right)\times\left(0, 2\right).
                \end{cases}
\end{equation*}
The discontinuity interface is the vertical line $x = \pi/3$. 

\textcolor{black}{In the case of two neurons ($n = 2$), we study two initializations of the nonlinear parameters $\br^{(0)}$: (a) one breaking line is the horizontal line $y = 0$ and the other is the vertical line $x = 0$; and (b) the two breaking lines are the vertical lines $x = 1/3$ and $x = 2/3$ which are on the left of the discontinuity interface $x=\pi/3$ (see \cref{fig1aa}). Both the initializations lead to local minima after $50$ iterations at where $\|\nabla_{\br}{\cal L}_{\tau}(u_{n, \tau})\|$ is smaller than $10^{-12}$, but the relative error is large (see \cref{table1}). In the case (a), since $\frac{\partial}{\partial y}\sigma_0(x,y)=\frac{\partial }{\partial y}\sigma_1(x,y)=0$, $\frac{\partial }{\partial y}\sigma_2(x,y)=1$, and $\sigma_2(x,0)=0$, 
the corresponding linear parameter $c_2$ of $\sigma_2(x,y)$ is zero, and hence this neuron does not move during the entire iteration. In the case (b), the SgGN method cannot move one of the vertical lines across the discontinuity interface. 
}

\begin{figure}[htbp!]
        \begin{subfigure}[b]{0.45\textwidth}
        \centering
        \includegraphics[width=\textwidth]{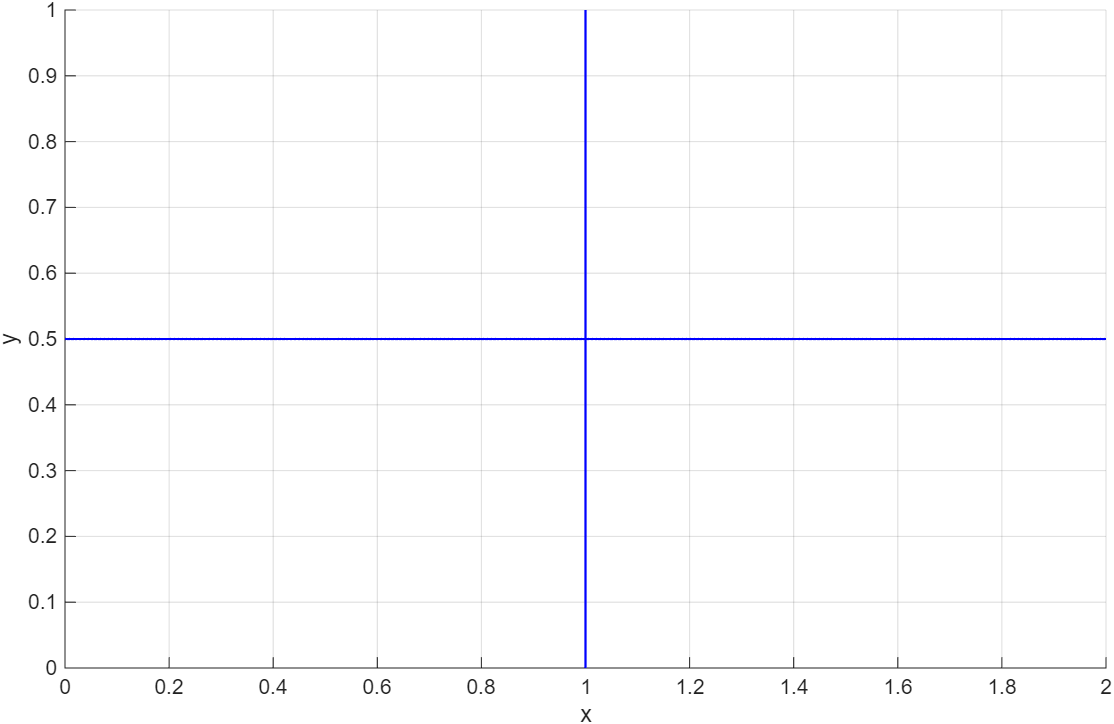}
    \end{subfigure}
    \hfill
    \begin{subfigure}[b]{0.45\textwidth}
        \centering
        \includegraphics[width=\textwidth]{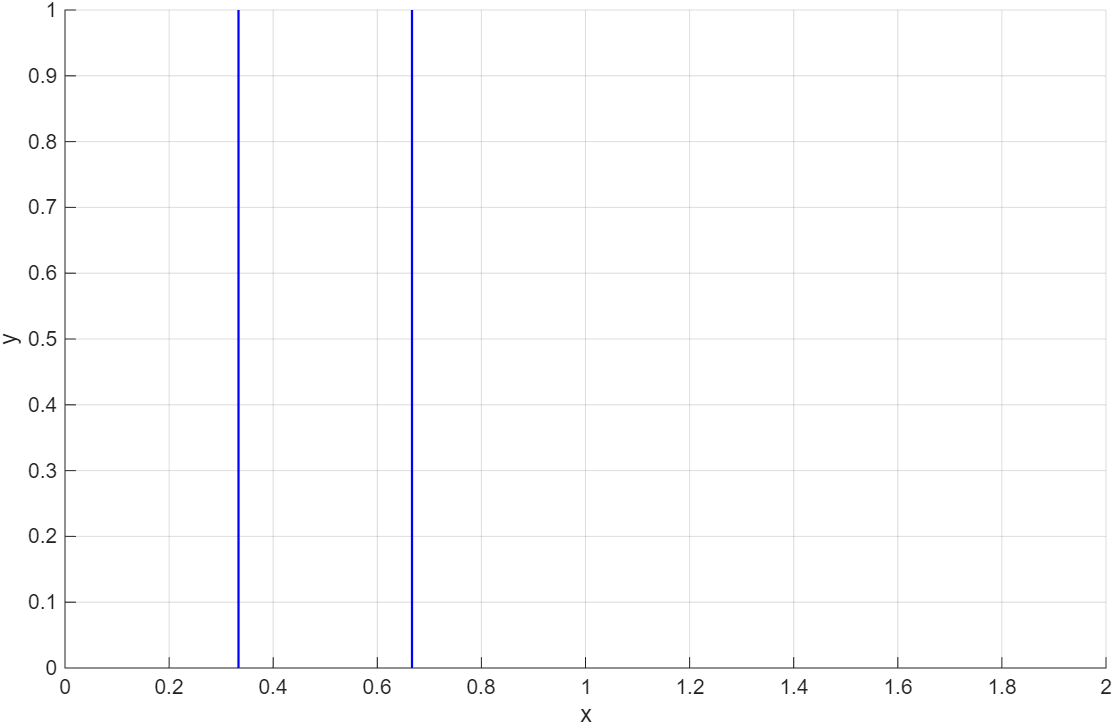}
    \end{subfigure}\\
                    \begin{subfigure}[b]{0.45\textwidth}
        \centering
        \includegraphics[width=\textwidth]{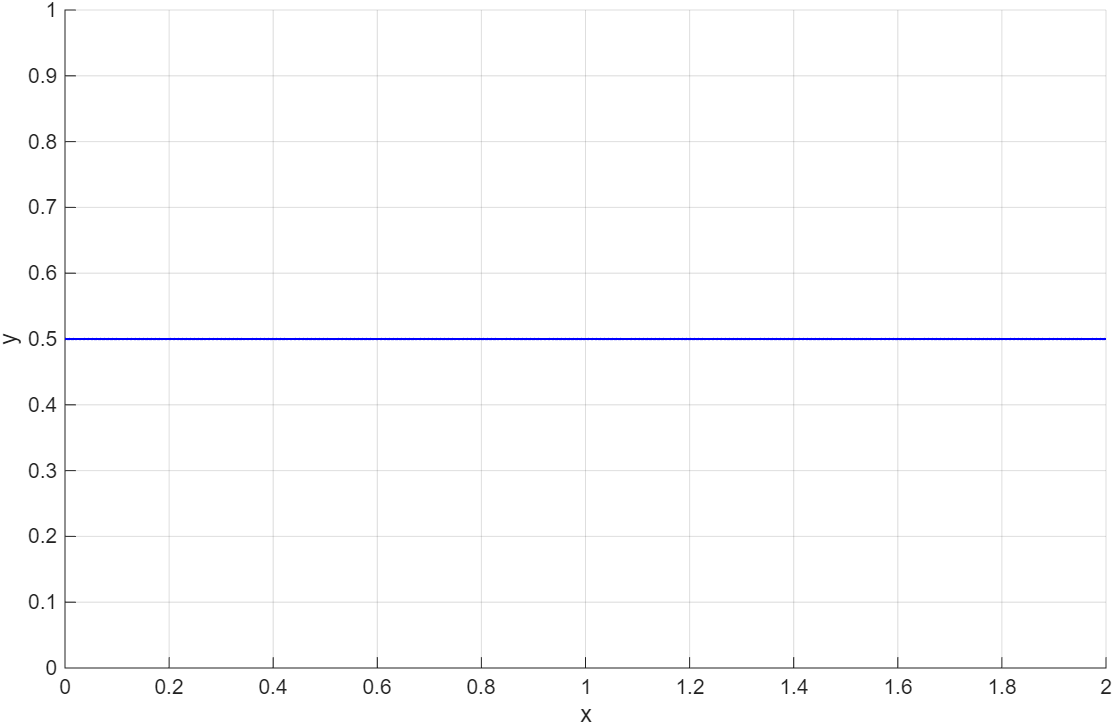}
    \end{subfigure}
    \hfill
    \begin{subfigure}[b]{0.45\textwidth}
        \centering
        \includegraphics[width=\textwidth]{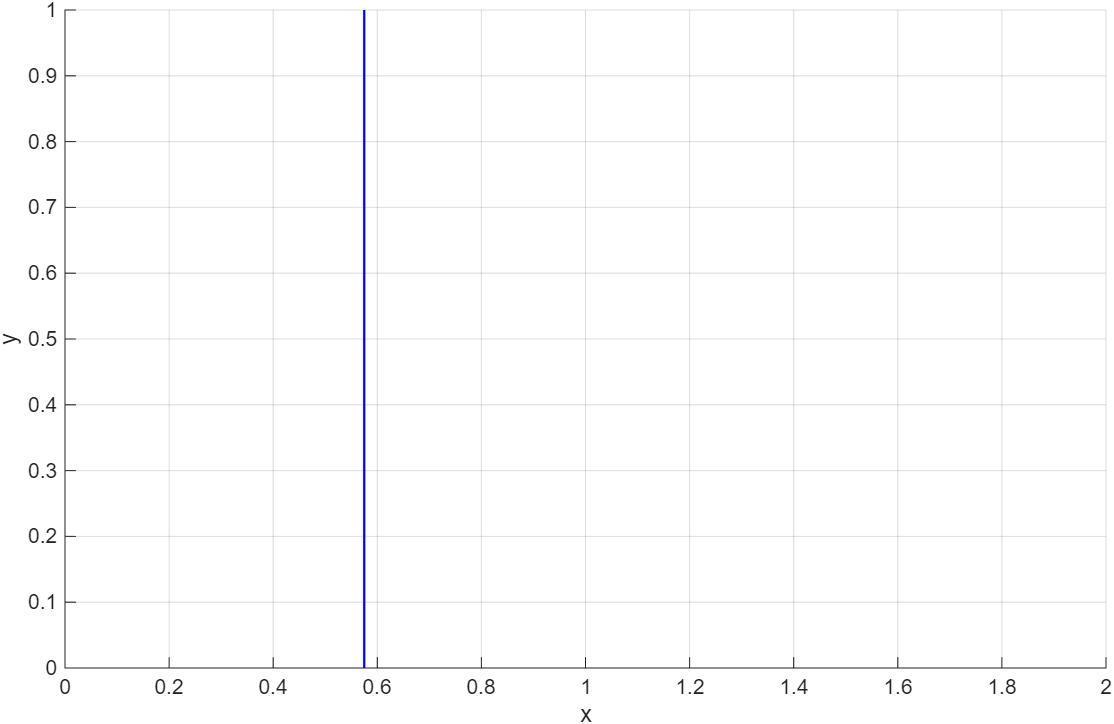}
    \end{subfigure}
    \caption{
    Top: initializations of the breaking lines. Bottom: the final breaking lines after 50 iterations, where $\|\nabla_{\br}{\cal L}(u_{n, \tau})\| = 1.07 \times 10^{-11}$ and $ 2.36 \times 10^{-10}$ on the respective left and right.}\label{fig1aa}
\end{figure}

The above experiments indicate need of two neurons initiated with vertical breaking lines containing the discontinuity interface. Next experiment has $n = 4$ neurons initiated with breaking lines to form a uniform partition of the domain $\Omega$. In the $L^2$ and energy norms, the SgGN method yields a relative error of $8.29\times 10^{-13}$ after only $50$ iterations (see \cref{table1}). Comparing to  the relative error of $5.80\times 10^{-2}$ achieved by $20,000$ iterations of the Adam optimizer in \cite{Cai2021linear}, this improvement is not only in the computational cost but more importantly in accuracy. 
For $n = 4$ neurons, \cref{fig1} depicts the neural network approximation, the traces of the exact solution and the NN approximation on $y=1$, and the final breaking lines. Again, two horizontal breaking lines are not moved \textcolor{black}{due to the fact that the corresponding linear parameters are zero}. 

\begin{table}[ht!]
    \centering
    \caption{Relative errors of the test problem with a vertical discontinuity interface after 50 iterations.}
    \begin{tabular}{c c c c}\label{table1}
        $n$  
        & $\dfrac{\|u - u_{n, \tau}\|_{0, \Omega}}{\|u\|_{0, \Omega}}$ 
        & $\dfrac{\|u - u_{n, \tau}\|_{{\bm \beta}, \Omega}}{\|u\|_{\bm \beta, \Omega}}$ 
         \\
        \hline
        2 (a)   & $3.63\times10^{-1}$  & $3.63\times10^{-1}$\\
        2 (b)   & $3.33\times10^{-1}$  & $3.33\times10^{-1}$
        \\
        4 & $8.29\times10^{-13}$  & $8.29\times10^{-13}$  \\
        \hline
    \end{tabular}
\end{table}

\begin{figure}[htbp!]
    \centering
    \begin{subfigure}[b]{0.31\textwidth}
        \centering
        \includegraphics[width=\textwidth]{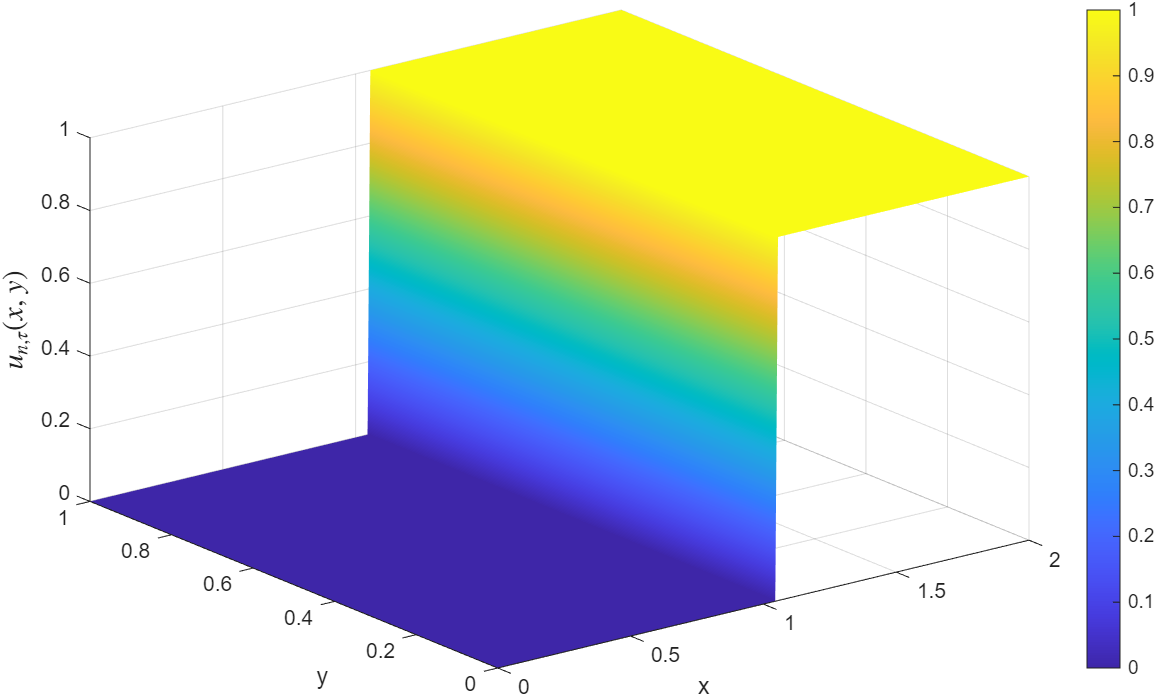}
        \caption{NN approximation $u_{n, \tau}$}\label{fig1a}
    \end{subfigure}
    \hfill
    \begin{subfigure}[b]{0.31\textwidth}
        \centering
        \includegraphics[width=\textwidth]{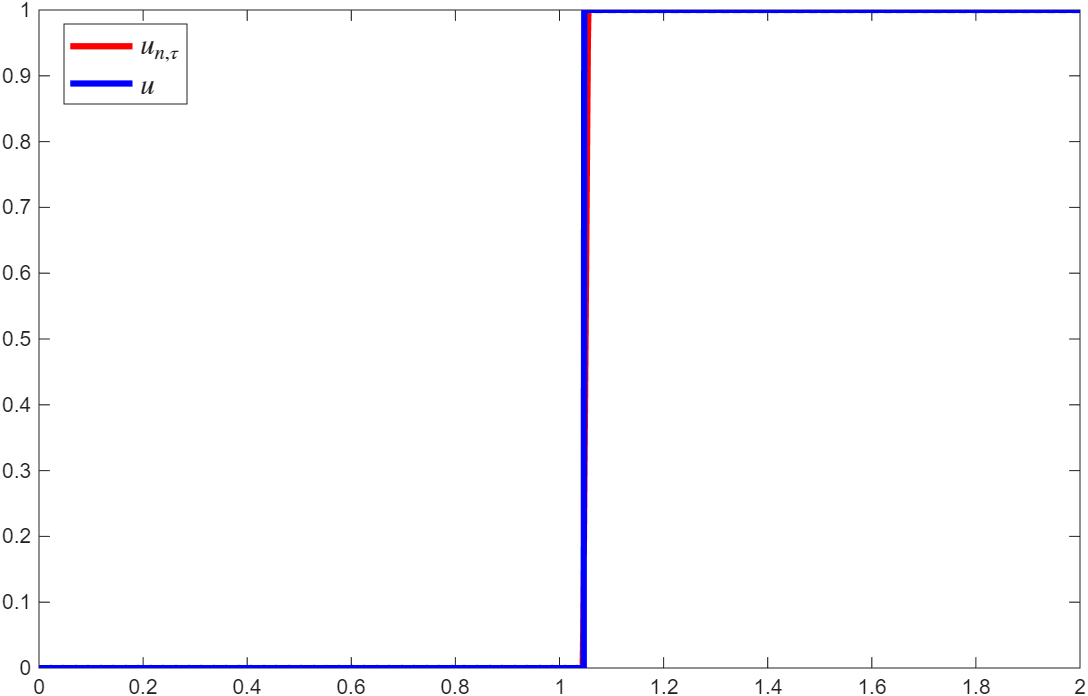}
        \caption{Traces on $y =1$}\label{fig1b}
    \end{subfigure}
    \hfill
    \begin{subfigure}[b]{0.31\textwidth}
        \centering
        \includegraphics[width=\textwidth]{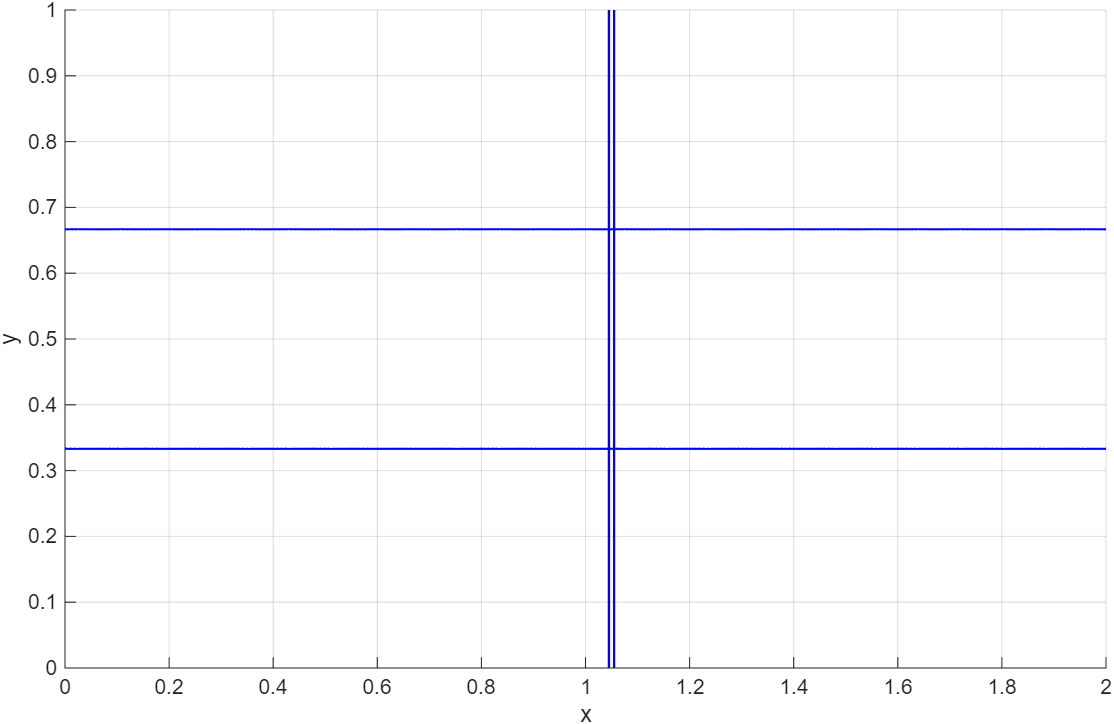}
        \caption{NN breaking lines}\label{fig1c}
    \end{subfigure}
    \caption{Results of $50$ SgGN iterations for $n = 4$ neurons. 
    }\label{fig1}
\end{figure}

\subsubsection{Diagonal Line Interface}
The second test problem is the equation in \cref{pde1} with the domain $\Omega = (-1, 1)^2$, ${\bm \beta} = (1/\sqrt{2}, 1/\sqrt{2})^T$, and $\gamma = 1$. The $\Gamma_{-} = \Gamma_{-}^1 \cup \Gamma_{-}^2$ is the inflow boundary, where $\Gamma_{-}^1 = \bigl\{(-1, y): y \in (-1, 1)\bigr\}$ and $\Gamma_{-}^{2} = \bigl\{(x, -1): x\in (-1, 1)\bigr\}$. 
The $f$ and $g$ and are piecewise constant given by
\begin{equation*}
                f(x, y)   = 
                \begin{cases}
                  1, &  (x, y) \in \Omega_1,  \\[2mm]
                  0, & (x, y) \in \Omega_2 
                \end{cases}  \quad \mbox{and} \quad g(x, y)   = 
                \begin{cases}
                  1, &  (x, y) \in \Gamma_{-}^1 ,  \\[2mm]
                  0, & (x, y) \in \Gamma^2_{-}, 
                \end{cases}  
\end{equation*}
where $\Omega_1 =\{(x, y) \in \Omega : y> x\}$ and $\Omega_2 = \{(x, y) \in \Omega : y< x\}$. The exact solution of this problem is $u(x, y) = f(x, y)$ with a discontinuity interface along the diagonal line $x = y$. 

\begin{figure}[htbp!]
    \centering
    \begin{subfigure}[b]{0.31\textwidth}
        \centering
        \includegraphics[width=\textwidth]{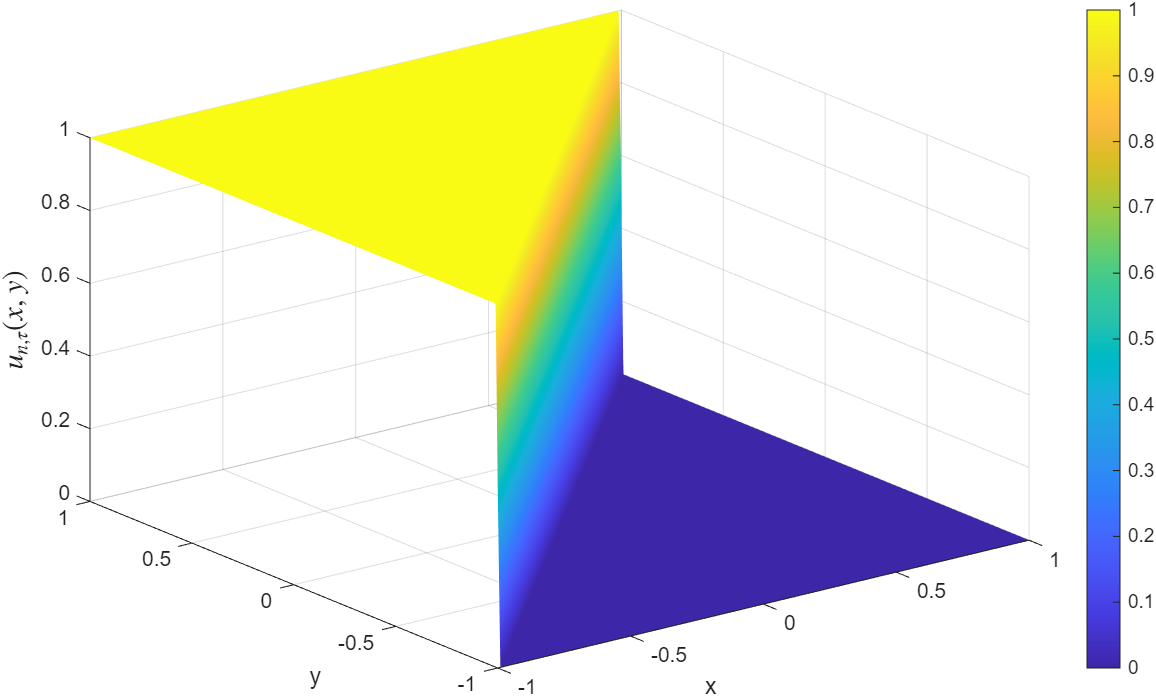}
        \caption{NN approximation $u_{n, \tau}$}\label{fig2a}
    \end{subfigure}
    \hfill
    \begin{subfigure}[b]{0.31\textwidth}
        \centering
        \includegraphics[width=\textwidth]{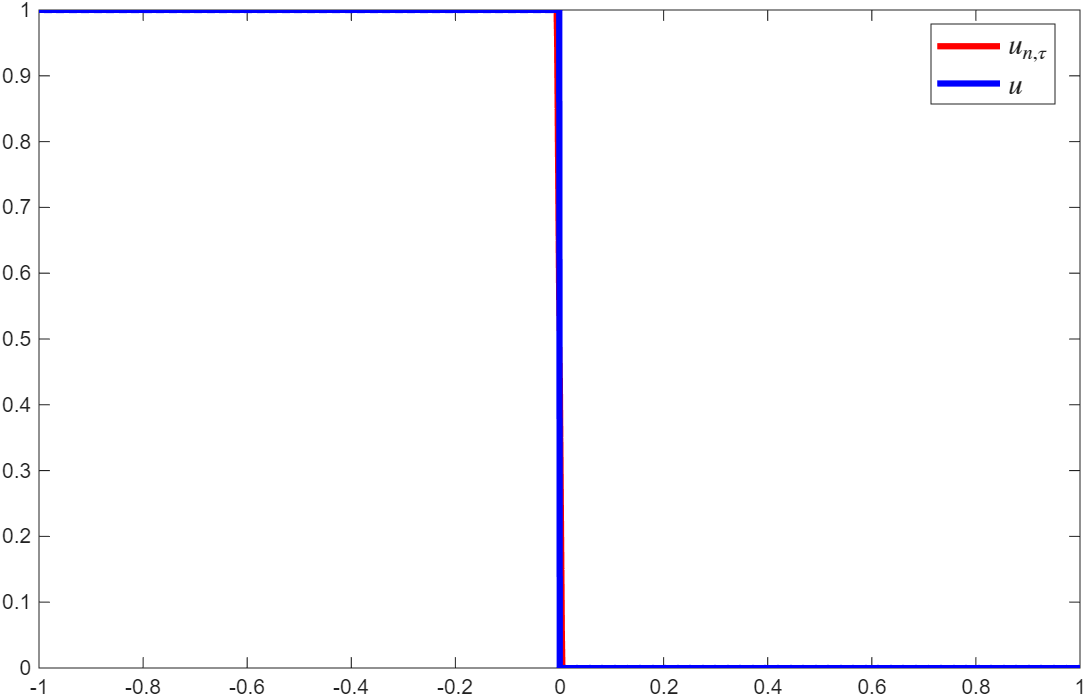}
        \caption{Traces on $y =-x$}\label{fig2b}
    \end{subfigure}
    \hfill
    \begin{subfigure}[b]{0.31\textwidth}
        \centering
        \includegraphics[width=\textwidth]{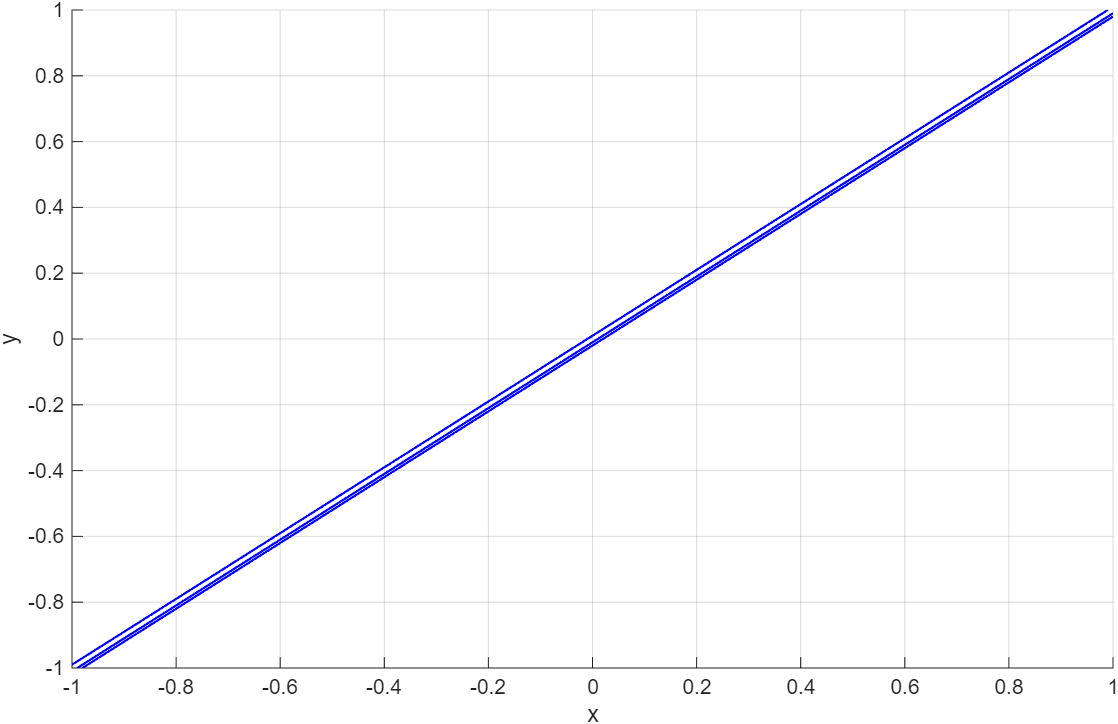}
        \caption{NN breaking lines}\label{fig2c}
    \end{subfigure}
    \caption{Results of $42$ SgGN iterations for $n = 4$ neurons. 
    }\label{fig2}
\end{figure}

For $n=4$ neurons, performance of the SgGN method is similar to the previous test problem: the relative error of $6.26\times 10^{-10}$ in the energy norm achieved by $42$ iterations (see \cref{table2}). Comparing to $20,000$ iterations of the Adam optimizer in \cite{Cai2021linear}, the Adam failed for $n=4$ neurons, and its accuracy is only $7.35\times 10^{-2}$ for $n=6$ neurons. For $n = 4$ neurons, using the stopping criterion $\mathcal{L}_{\tau}(u_{n, \tau}; \bff) \leq  2 \times 10^{-18}$ ($42$ SgGN iterations), the neural network approximation depicted in \cref{fig2a} is super-accurate and free of oscillations. Furthermore, the traces along $y = -x$ shown in \cref{fig2b} demonstrate that the numerical approximation captures the discontinuity accurately. The corresponding breaking lines are plotted in \cref{fig2c}, which are located near the diagonal interface.

Due to dependence on initialization, a natural question is: would increasing the number of neurons reduce the number of the SgGN iterations or improve accuracy if initializing them uniformly? For the stopping criterion $\mathcal{L}_{\tau}(u_{n, \tau}; \bff) \leq 2 \times 10^{-18}$ and $\mathcal{L}_{\tau}(u_{n, \tau}; \bff) \leq 10^{-8}$, \cref{table2} and \cref{table22} show that the answer is negative. Moreover, the resulting relative errors in the energy norm deteriorate for larger values of $n$. However, when a fixed number of 100 iterations is used instead, these relative errors improve, as demonstrated in \cref{table2b}. For this test problem, these findings show that adding more neurons complicates the optimization process and does not necessarily provide significant improvement in accuracy. Furthermore, establishing an effective stopping criterion remains generally difficult.



\begin{table}[ht!]
    \centering
    \caption{Loss functional values and relative errors of the problem with discontinuity along the diagonal,  using the stopping criterion $\mathcal{L}_{\tau}(u_{n, \tau}; \bff) \leq  2 \times 10^{-18}$.}
    \begin{tabular}{c c c c c c}\label{table2}
        $n$ 
        & $\mathcal{L}_{\tau}\big(u_{n, \tau};{\bf f}\big)$ &$\dfrac{\|u - u_{n, \tau}\|_{0, \Omega}}{\|u\|_{0, \Omega}}$ 
        & $\dfrac{\|u - u_{n, \tau}\|_{{\bm \beta}, \Omega}}{\|u\|_{\bm \beta, \Omega}}$ & Iterations
         \\
        \hline
        4 &  $1.69\times 10^{-18}$& $6.58\times10^{-11}$  & $6.26\times10^{-10}$ & $42$ \\
        8 & $6.48\times 10^{-19}$&$2.33\times10^{-11}$  & $7.19\times10^{-10}$ & $42$ \\
        12 & $1.99\times10^{-18}$&$2.64\times10^{-11}$  & $1.20\times10^{-9}$ & $54$ \\
        16 & $1.97\times10^{-18}$ & $9.71\times10^{-11}$  & $1.46\times10^{-5}$ & $21$ \\
        \hline
    \end{tabular}
\end{table}

\begin{table}[ht!]
    \centering
    \caption{Loss functional values and relative errors of the problem with discontinuity along the diagonal, using the stopping criterion $\mathcal{L}_{\tau}(u_{n, \tau}; \bff) \leq 10^{-8}$.}
    \begin{tabular}{c c c c c c}\label{table22}
        $n$ 
        & $\mathcal{L}_{\tau}\big(u_{n, \tau};{\bf f}\big)$ &$\dfrac{\|u - u_{n, \tau}\|_{0, \Omega}}{\|u\|_{0, \Omega}}$ 
        & $\dfrac{\|u - u_{n, \tau}\|_{{\bm \beta}, \Omega}}{\|u\|_{\bm \beta, \Omega}}$ & Iterations
         \\
        \hline
        4 &  $2.70\times 10^{-9}$& $4.17\times10^{-6}$  & $4.20\times10^{-6}$ & $35$ \\
        8 & $6.48\times 10^{-19}$&$2.33\times10^{-11}$  & $7.19\times10^{-10}$ & $42$ \\
        12 & $8.07\times10^{-18}$&$8.57\times10^{-10}$  & $1.29\times10^{-7}$ & $49$ \\
        16 & $1.35\times10^{-12}$ & $4.36\times10^{-7}$  & $8.84\times10^{-4}$ & $17$ \\
        \hline
    \end{tabular}
\end{table}

\begin{table}[ht!]
    \centering
    \caption{Loss functional values and relative errors of the problem with discontinuity along the diagonal after 100 iterations.}
    \begin{tabular}{c c c c }\label{table2b}
        $n$ 
        & $\mathcal{L}_{\tau}\big(u_{n, \tau};{\bf f}\big)$ &$\dfrac{\|u - u_{n, \tau}\|_{0, \Omega}}{\|u\|_{0, \Omega}}$ 
        & $\dfrac{\|u - u_{n, \tau}\|_{{\bm \beta}, \Omega}}{\|u\|_{\bm \beta, \Omega}}$\\
        \hline
        4 &  $1.57\times 10^{-18}$& $6.25\times10^{-11}$  & $6.52\times10^{-10}$ \\
        8 & $4.80\times 10^{-19}$&$5.99\times10^{-12}$  & $5.91\times10^{-10}$ \\
        12 & $2.40\times10^{-18}$&$6.40\times10^{-11}$  & $1.17\times10^{-9}$ \\
        16 & $6.92\times10^{-18}$ & $5.78\times10^{-10}$  & $1.07\times10^{-9}$\\
        \hline
    \end{tabular}
\end{table}

\subsection{Problem with Piecewise Smooth Solution}
The third test problem is discontinuous along the diagonal line $x = y$, and its solution is piecewise smooth. Specifically, 
${\bm \beta} = (1/\sqrt{2}, 1/\sqrt{2})^T$, $\gamma =f = 0$, $\Omega = (0, 1)^2$, and $\Gamma_{-} = \Gamma_{-}^1 \cup \Gamma_{-}^2$, where $\Gamma_{-}^1 = \bigl\{(0, y): y \in (0, 1)\bigr\}$ and $\Gamma_{-}^{2} = \bigl\{(x, 0): x\in (0, 1)\bigr\}$. The inflow data $g$ and the exact solution $u$ are given by
\begin{equation*}
                g(x, y)   = 
                \begin{cases}
                  \sin(y), &  (x, y) \in\Gamma_{-}^1 ,  \\[2mm]
                  \cos(x), & (x, y) \in \Gamma_{-}^2
                \end{cases}
   \quad \mbox{and} \quad u(x, y)   = 
                \begin{cases}
                  \sin(y-x), &  (x, y) \in \bigl\{(x, y) \in \Omega: y > x\bigr\},  \\[2mm]
                  \cos(x-y), & (x, y) \in \bigl\{(x, y) \in \Omega: y < x\bigr\}.
                \end{cases}
\end{equation*}

The exact solution of this problem is depicted in \cref{fig3a}, along with the resulting approximation after 25 SgGN iterations for $n = 24$. The traces of the exact and numerical solutions along the plane $y = 1 - x$ are illustrated in \cref{fig3c}, which exhibit no oscillations. \cref{table3} and \cref{table3b} present the loss functional values and relative errors in both the $L^2$ and energy norms for various $n$ (the number of neurons). As $n$ increases, the values of the functional and the relative error decrease, and the convergence rate of approximation is much better than the linear rate.
Compared to the Adam optimizer in \cite{Cai2021linear}, $30,000$ iterations of the Adam optimizer can only achieve $1.01\times 10^{-1}$ accuracy for $n=40$ neurons, while the accuracy of 25 SgGN iterations is already $1.70\times 10^{-3}$ for $n=12$ neurons and $1.97\times 10^{-4}$ for $n=36$ neurons (see \cref{table3}).

\begin{figure}[htbp!]
    \centering
    \begin{subfigure}[b]{0.45\textwidth}
        \centering
        \includegraphics[width=\textwidth]{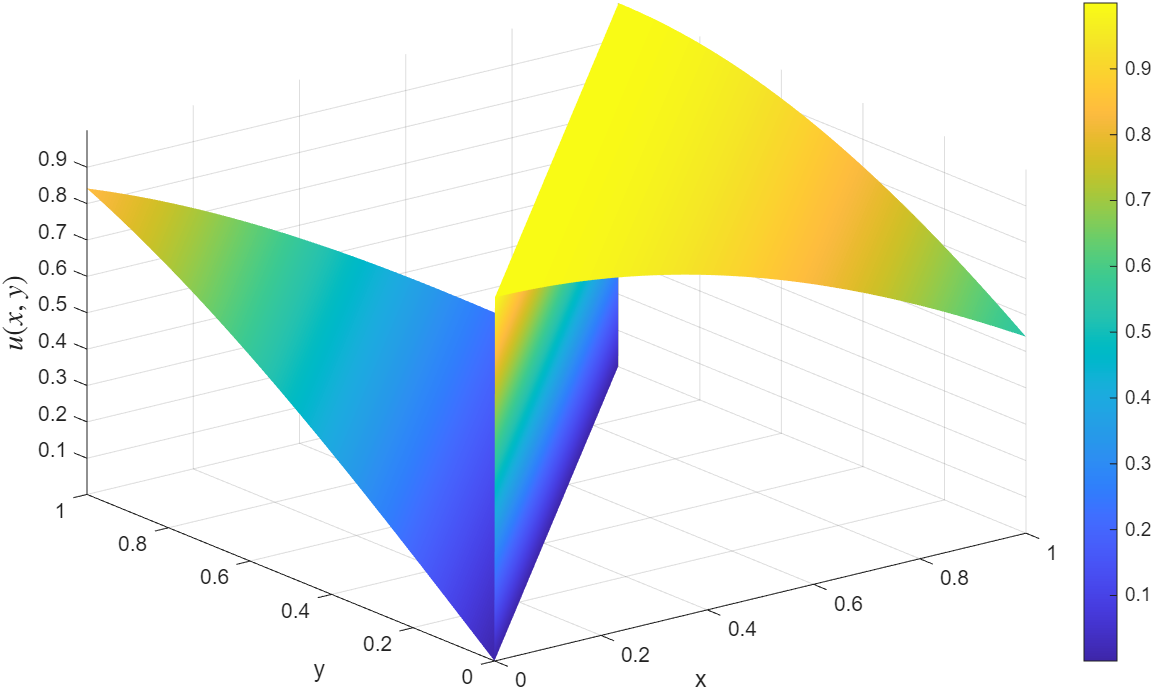}
        \caption{Exact solution}\label{fig3a}
    \end{subfigure}
    \hfill
    \begin{subfigure}[b]{0.45\textwidth}
        \centering
        \includegraphics[width=\textwidth]{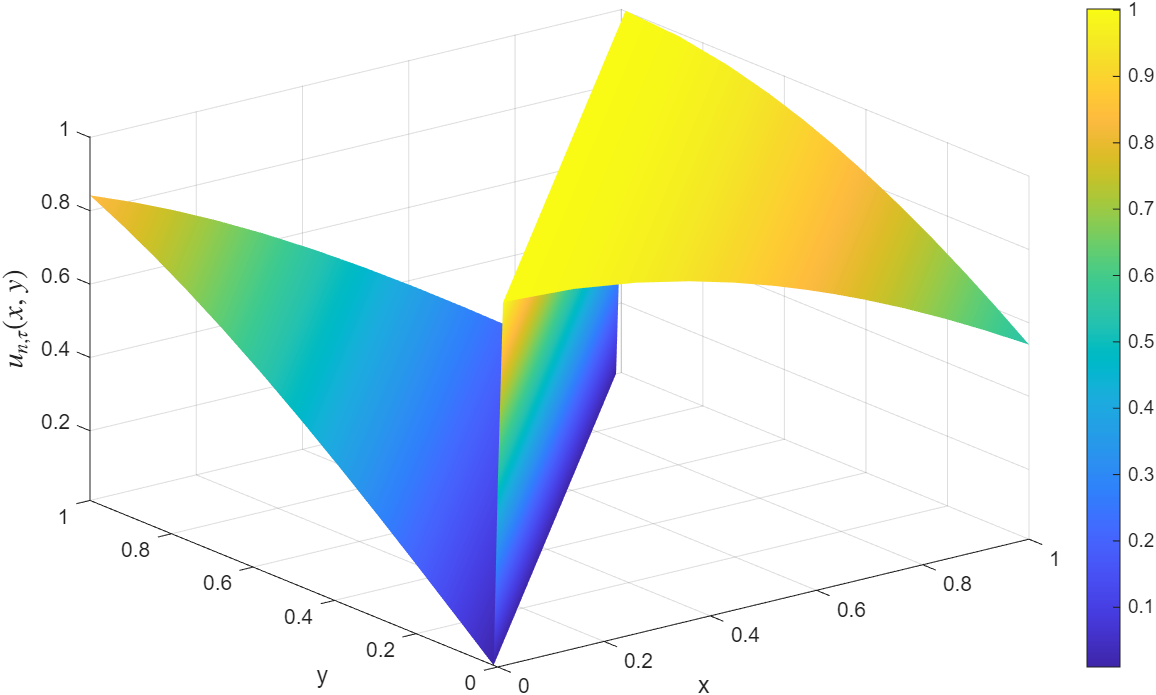}
        \caption{NN approximation $u_{n, \tau}$}\label{fig3b}
    \end{subfigure}\\
        \begin{subfigure}[b]{0.45\textwidth}
        \centering
        \includegraphics[width=\textwidth]{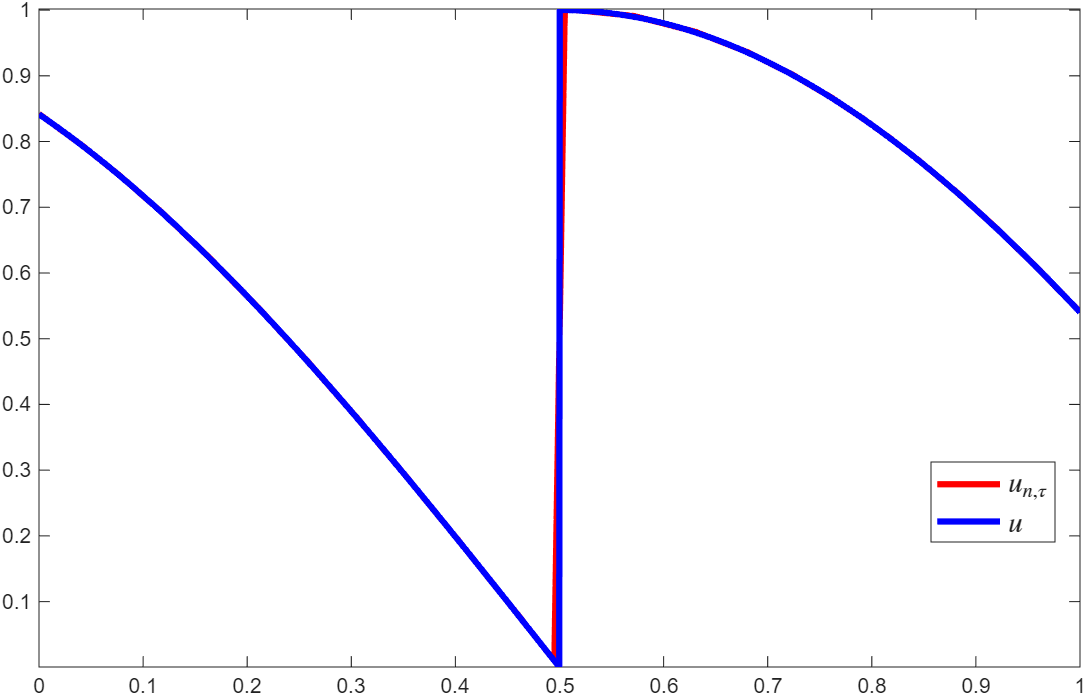}
        \caption{Traces on $y = 1-x$}\label{fig3c}
    \end{subfigure}
    \hfill
    \begin{subfigure}[b]{0.45\textwidth}
        \centering
        \includegraphics[width=\textwidth]{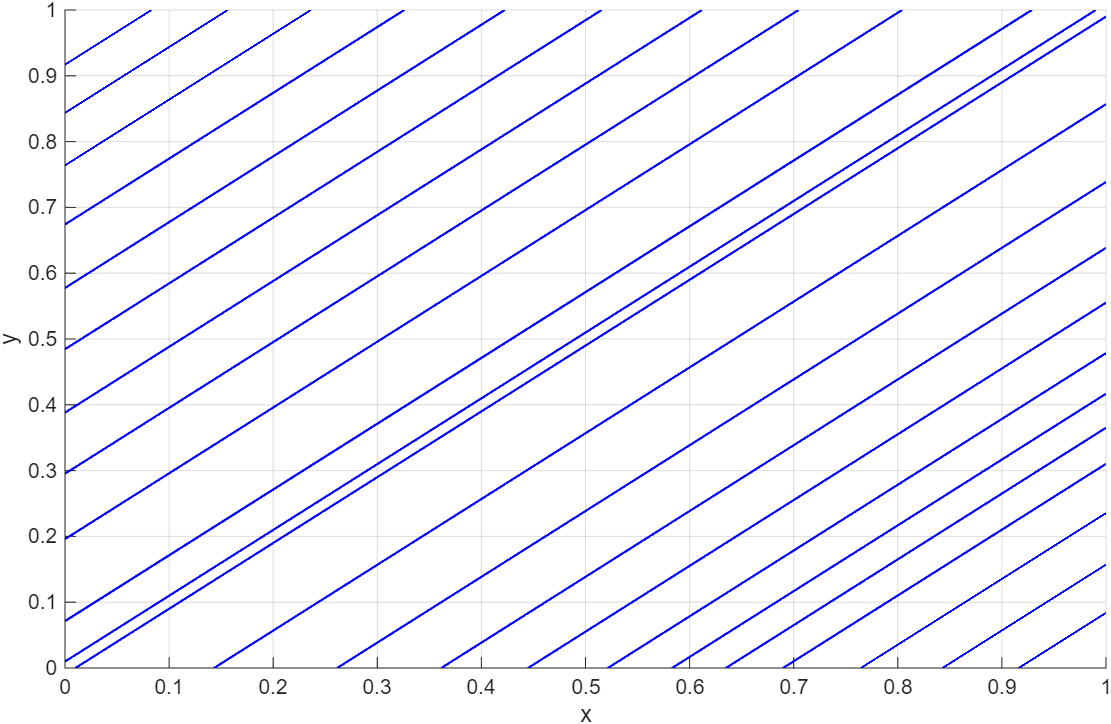}
        \caption{NN breaking lines}\label{fig3d}
    \end{subfigure}
    \caption{Results of $25$ SgGN iterations for $n = 24$ neurons. 
    }\label{fig3}
\end{figure}

Numerical results reported in \cref{table3} and \cref{table3b} are obtained after 25 SgGN iterations and used two different initializations of the nonlinear parameters, respectively. The former starts with a uniform partition of $\Omega$ by the breaking lines, and the latter places the breaking lines parallel to the advection field $\bm{\beta}$ and uniformly partitioning the inflow boundary $\Gamma_{-}$. The geometrical meaning of the SgGN method for the nonlinear parameters is to rotate and shift the breaking lines to the right locations. In the case of the latter, since the initial positions of the breaking lines are parallel to the final positions, the SgGN method only shifts the breaking lines. Because both the initializations achieve similar accuracy (see \cref{table3} and \cref{table3b}), the SgGN method is very effective not only in shifting but also in rotating the breaking lines. This is further confirmed by the fact that the majority of neurons remained active in both cases. Here a neuron is called active if its linear parameter is greater than $\epsilon_c$ (as defined in \cref{setI}).



\begin{table}[ht!]
    \centering
    \caption{Loss functional values and relative errors for the problem with a piecewise smooth solution after 25 iterations, using an initial uniform partition.}
    \begin{tabular}{c c c c c}\label{table3}
        $n$ 
        & Active neurons&$\mathcal{L}_{\tau}\big(u_{n, \tau};{\bf f}\big)$& $\dfrac{\|u - u_{n, \tau}\|_{0, \Omega}}{\|u\|_{0, \Omega}}$ 
        & $\dfrac{\|u - u_{n, \tau}\|_{{\bm \beta}, \Omega}}{\|u\|_{\bm \beta, \Omega}}$ 
         \\
        \hline
        12   & 12 & $4.74\times10^{-4}$& $1.70\times10^{-3}$  & $1.70\times10^{-3}$  \\
        24 & 24 & $2.58\times10^{-5}$&$4.30\times10^{-4}$  & $4.30\times10^{-4}$  \\
        36 & 36 & $5.09\times10^{-6}$& $1.97\times10^{-4}$  & $1.97\times10^{-4}$  \\
        48 & 47 & $3.53\times10^{-6}$&$1.57\times10^{-4}$  & $1.57\times10^{-4}$  \\
        60 & 60 & $1.42\times10^{-6}$& $7.54\times10^{-5}$  & $7.54\times10^{-5}$  \\
        \hline
    \end{tabular}
\end{table}

\begin{table}[ht!]
    \centering
    \caption{Loss functional values and relative errors for the problem with a piecewise smooth solution after 25 iterations, using an initial partition along the advection direction.}
    \begin{tabular}{c c  c c c}\label{table3b}
        $n$ 
        & Active neurons& $\mathcal{L}_{\tau}\big(u_{n, \tau};{\bf f}\big)$ & $\dfrac{\|u - u_{n, \tau}\|_{0, \Omega}}{\|u\|_{0, \Omega}}$ 
        & $\dfrac{\|u - u_{n, \tau}\|_{{\bm \beta}, \Omega}}{\|u\|_{\bm \beta, \Omega}}$ 
         \\
        \hline
        12   & 12 & $3.28 \times 10^{-4}$&$1.36\times10^{-3}$  & $2.68\times10^{-3}$  \\
        24 & 24 & $3.32\times 10^{-5}$& $3.76\times10^{-4}$  & $3.76\times10^{-4}$  \\
        36 & 36 &  $7.78 \times  10^{-6}$&$1.52\times10^{-4}$  & $1.52\times10^{-4}$  \\
        48 & 48 & $1.13 \times 10^{-6}$& $8.52\times10^{-5}$  & $8.52\times10^{-5}$  \\
        60 & 60  & $1.04\times10^{-6}$& $8.14\times10^{-5}$  & $8.14\times10^{-5}$  \\
        \hline
    \end{tabular}
\end{table}

\subsection{Problem with Two Discontinuity Interfaces}
The fourth test problem is piecewise smooth with two discontinuity interfaces. Specifically, ${\bm \beta} = (1/\sqrt{2}, 1/\sqrt{2})^T$, $\gamma = 1$, the domain $\Omega = (-1, 1) \times (0, 1)$, and the inflow boundary $\Gamma_{-} = \bigl\{(x, 0):x\in [-1, 1) \bigr\} \cup\bigl\{(-1, y):y\in(0, 1) \bigr\}$. The right-hand side $f$ and the inflow boundary data $g$ are given by 
\begin{equation*}
                f(x, y)   = 
                \begin{cases}
                  \sin\left(\dfrac{\pi(x-y+0.9)}{0.3}\right), &  (x, y) \in \Upsilon_1  = \bigl\{(x, y) \in \Omega: -0.9 < x-y< -0.6\bigr\},  \\[2mm]
                  -1, & (x, y) \in \Upsilon_2 = \bigl\{(x, y)\in  \Omega: -0.2<x -y<0.1\bigr\},
                  \\[2mm]
                  0, &(x, y) \in  \Omega\backslash\left(\Upsilon_1 \cup \Upsilon_2\right)
                \end{cases}
\end{equation*}
and 
\begin{equation*}
                g(x, y)   = 
                \begin{cases}
                  \sin\left(\dfrac{\pi(x-y+0.9)}{0.3}\right), &  (x, y) \in \Gamma_{-}^1 = \bigl\{(x, 0): x \in (-0.9, -0.6)\bigr\},  \\[2mm]
                  -1, & (x, y) \in \Gamma_{-}^2 = \bigl\{(x, 0): x \in (-0.2, 0.1)\bigr\},
                  \\[2mm]
                  0, &(x, y) \in \Gamma_{-} \backslash\left(\Gamma_{-}^1 \cup \Gamma_{-}^2\right),
                \end{cases}
\end{equation*}
respectively. \cref{pde1} has the exact solution $u = f$.

For $n = 30$ neurons \textcolor{black}{uniformly partitioning the domain $\Omega$}, \cref{fig4b} depicts the NN approximation after $100$ iterations of the SgGN method. The traces of the NN approximation and the exact solution on the plane $y = 0.8$ are shown in \cref{fig4c}. Similar to the other test problems, the resulting approximation exhibits no overshooting while using relatively few degrees of freedom. \cref{table4} again shows that the values of the functional and the relative error decrease as $n$ increases, with a convergence rate of approximation much better than the linear rate. Moreover, the number of active neurons after $100$ iterations is significantly less than that at the initialization. Finally, comparing to $80,000$ Adam iterations reported in \cite{Cai2021linear} (Table~4), the Adam can only achieve the accuracy of $1.20\times 10^{-1}$ with $n=40$ neurons, which is much worse than that of $100$ SgGN iterations even with only $n=15$ neurons (see \cref{table4}). 


Comparing to the third test problem, the accuracy of this problem for the active neurons is slightly worse. This observation motivates \cref{table4n}, where we report the errors in two separate regions. In the region $\Upsilon_1$, the solution $\sin\left(\frac{\pi(x - y + 0.9)}{0.3}\right)$ is smooth and complicated, and hence the error is mainly in this region due to the small number of neurons. In the rest of the region $\Omega \setminus \Upsilon_1$, the solution is piecewise constant, and the NN approximation is extremely accurate as reported in the previous test problems. 



\begin{table}[ht]
    \centering
    \caption{Relative errors of the problem with two discontinuous interfaces after 100 iterations.}
    \begin{tabular}{c c c c c c}\label{table4}
        $n$ & Active neurons 
        & $\mathcal{L}_{\tau}\big(u_{n, \tau};{\bf f}\big)$& $\dfrac{\|u - u_{n, \tau}\|_{0, \Omega}}{\|u\|_{0, \Omega}}$ 
        & $\dfrac{\|u - u_{n, \tau}\|_{{\bm \beta}, \Omega}}{\|u\|_{\bm \beta, \Omega}}$ 
         \\
        \hline
         15  & 11& $6.31\times 10^{-2}$& $2.44\times10^{-2}$   & $2.44\times10^{-2}$  \\
     20 & 12&$7.54\times 10^{-3}$ &$7.45\times10^{-3}$  & $7.45\times10^{-3}$  \\
          30 &21 & $2.48 \times 10^{-3}$&$2.80\times10^{-3}$  & $2.80\times10^{-3}$ \\
        \hline
    \end{tabular}
\end{table}


\begin{table}[ht]
    \centering
    \caption{Numerical errors across different regions of $\Omega$ after 100 iterations for the problem with two discontinuous interfaces.}
    \begin{tabular}{c c c c c c}\label{table4n}
        $n$ & $\|u - u_{n, \tau}\|_{0, \Upsilon_1}$ 
        & $\|u - u_{n, \tau}\|_{{\bm \beta}, \Upsilon_1}$& $\|u - u_{n, \tau}\|_{0, \Omega \backslash\Upsilon_1}$ 
        & $\|u - u_{n, \tau}\|_{{\bm \beta}, \Omega \backslash\Upsilon_1}$ 
         \\
        \hline
         15  & $1.62\times 10^{-2}$& $1.62\times 10^{-2}$& $0$  & $3.29\times10^{-10}$  \\
     20 & $4.92\times 10^{-3}$&$4.96\times 10^{-3}$ &$0$  & $1.25\times10^{-9}$  \\
          30 &$1.86\times10^{-3}$& $1.86\times10^{-3}$& $0$  & $5.39\times10^{-10}$ \\
        \hline
    \end{tabular}
\end{table}


\begin{figure}[htbp!]
    \centering
    \begin{subfigure}[b]{0.45\textwidth}
        \centering
        \includegraphics[width=\textwidth]{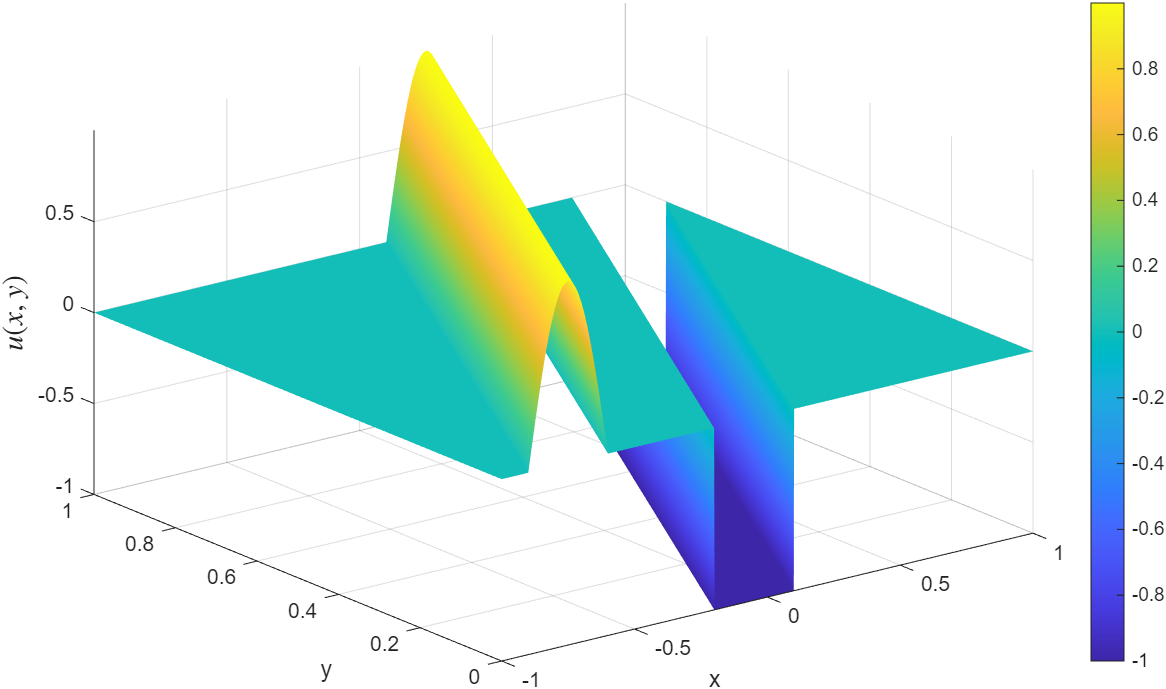}
        \caption{Exact Solution}\label{fig4a}
    \end{subfigure}
    \hfill
    \begin{subfigure}[b]{0.45\textwidth}
        \centering
    \includegraphics[width=\textwidth]{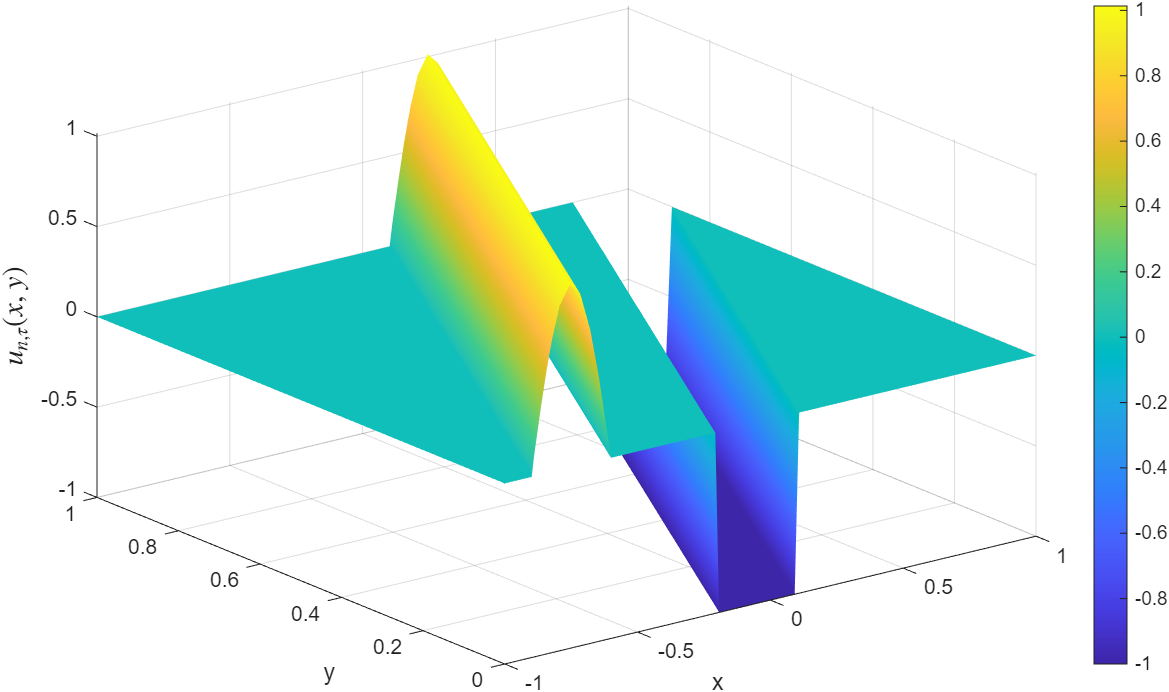}
        \caption{NN Approximation}
        \label{fig4b}
    \end{subfigure}\\
        \begin{subfigure}[b]{0.45\textwidth}
        \centering
        \includegraphics[width=\textwidth]{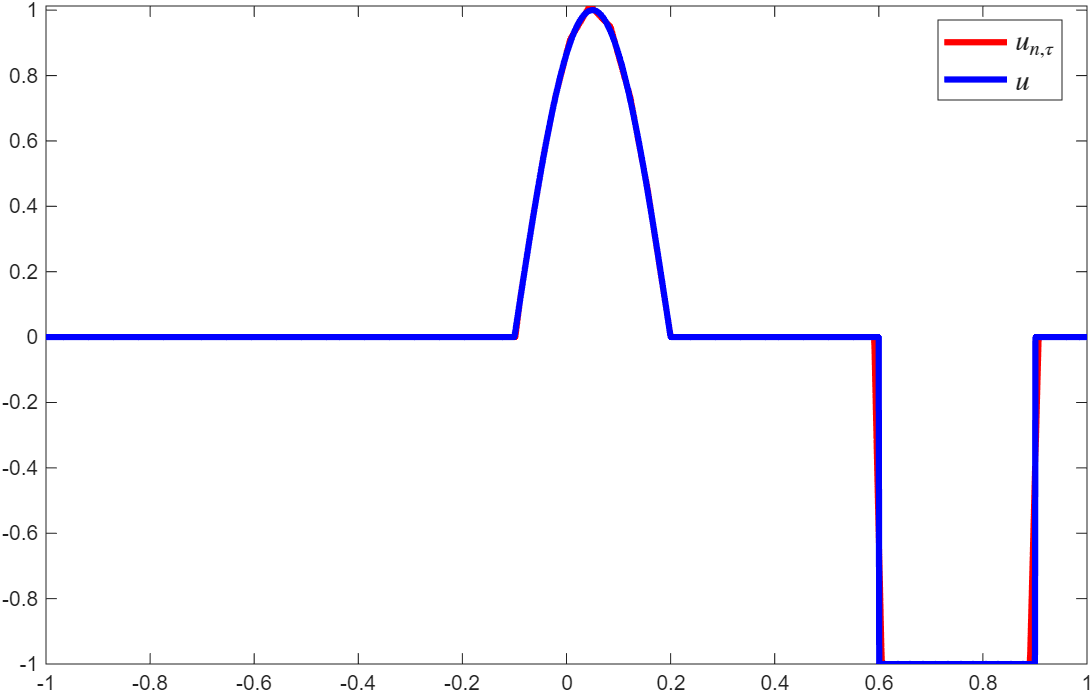}
        \caption{Traces on $y = 0.8$}\label{fig4c}
    \end{subfigure}
    \hfill
    \begin{subfigure}[b]{0.45\textwidth}
        \centering
        \includegraphics[width=\textwidth]{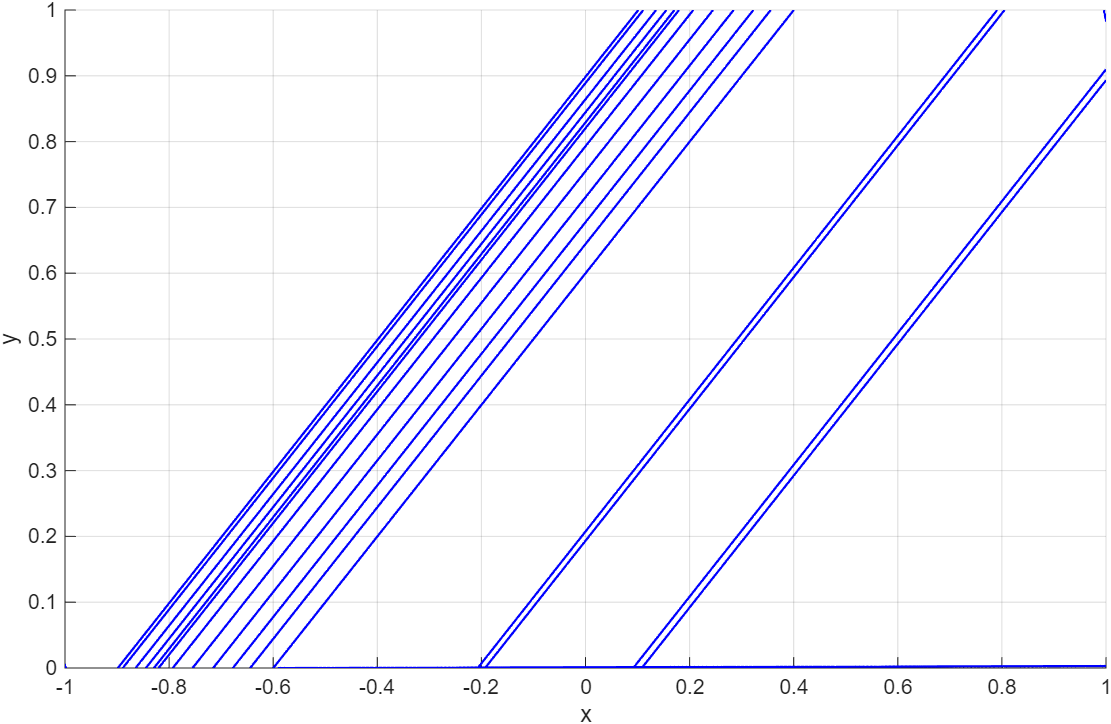}
        \caption{NN breaking lines}\label{fig4d}
    \end{subfigure}
    \caption{Results of $100$ SgGN iterations for $n = 30$ neurons. 
    }\label{example2DR}
\end{figure}

\section{Conclusions}\label{sec:conclusion}

A SgGN method was derived for solving the resulting non-convex optimization problem arising from the LSNN method for the two-dimensional linear advection-reaction equation when using shallow ReLU neural networks. Extension of the method to multi-dimension is straightforward.
The method was implemented for all test problems in \cite{Cai2021linear}. Numerical experiments demonstrate that the SgGN method is much more efficient and accurate than the Adam optimizer. First, the SgGN method can achieve the same accuracy with much less computational cost than the Adam. While the Adam stuck at $10^{-2}$ or $10^{-1}$ accuracy after $20,000$ to $80,000$ iterations, the SgGN achieves $10^{-10}$ to $10^{-5}$ accuracy after $100$ iterations. Geometrically, the SgGN can rotate and shift the breaking lines very efficiently. 

However, due to the ill-condition of the stiffness and GN matrices, we need to use the truncated SVD to solve the corresponding algebraic systems at each iteration step. This is reasonably expensive, but it is compensated by the fact that neural networks use significantly fewer degrees of freedom than standard mesh-based numerical methods.

\bigskip
\bibliographystyle{ieee}
\bibliography{Reference}

\end{document}

%% file: ex_shared.tex
\usepackage{lipsum}
\usepackage{amsfonts}
\usepackage{graphicx}
\usepackage{epstopdf}
\usepackage{algorithmic}
\ifpdf
  \DeclareGraphicsExtensions{.eps,.pdf,.png,.jpg}
\else
  \DeclareGraphicsExtensions{.eps}
\fi

\newsiamremark{hypothesis}{Hypothesis}
\crefname{hypothesis}{Hypothesis}{Hypotheses}
\newsiamthm{claim}{Claim}

\title{Structure-Guided Gauss-Newton Method: \\ Linear Advection-Reaction Equation}

\author{Zhiqiang Cai\thanks{Department of Mathematics, Purdue University, West Lafayette, IN
    (\email{caiz@purdue.edu},  \email{herre125@purdue.edu}).}
\and César Herrera\footnotemark[1]
}

\usepackage{amsopn}